\documentclass[
  a4paper, 
  reqno, 
  oneside, 
  12pt
]{amsart}

\usepackage[utf8]{inputenc}

\usepackage[dvipsnames]{xcolor} 
\colorlet{cite}{black}
\usepackage{tikz}  
\usetikzlibrary{arrows,positioning}
\tikzset{ 
  baseline=-2.3pt,
  text height=1.5ex, text depth=0.25ex,
  >=stealth,
  node distance=2cm,
  mid/.style={fill=white,inner sep=2.5pt},
}

\usepackage{enumitem}
\usepackage{lmodern}
\usepackage{tikz-cd}
\usepackage{amsthm, amssymb, amsfonts}
\usepackage{placeins}
\usepackage{amsthm, amssymb, amsfonts}	
\usepackage[all]{xy}

\usepackage{graphicx,caption,subcaption}
\usepackage{braket}  
\usepackage{microtype}

\usepackage[makeroom]{cancel}
\usepackage[%
  bookmarks=true,			
  unicode=true,			
  pdftitle={Invariant generalized complex structures on real flags},		%
  pdfauthor={San Martin}{Valencia}{Varea},	%
  pdfkeywords={Invariant generalized geometry},	
  colorlinks=true,		
  linkcolor=black,			
  citecolor=cite,		
  filecolor=Red,		
  urlcolor=true			
]{hyperref}				
\usepackage{cleveref}

\newtheoremstyle{mydef}
  {}		
  {}		
  {}		
  {}		
  {\scshape}	
  {. }		
  { }		
  {\thmname{#1}\thmnumber{ #2}\thmnote{ #3}}	

\newtheorem{theorem}{Theorem}[section]
\newtheorem*{theorem*}{Theorem}
\newtheorem{proposition}[theorem]{Proposition}
\newtheorem*{proposition*}{Proposition}

\newtheorem*{lemma*}{Lemma}

\newtheorem*{corollary*}{Corollary}
\theoremstyle{definition}
\newtheorem{definition}[theorem]{Definition}
\newtheorem{example}[theorem]{Example}
\newtheorem{notation}[theorem]{Notation}
\theoremstyle{remark}
\newtheorem{remark}[theorem]{Remark}

\newtheorem*{conjecture*}{Conjecture}
\usepackage{tikz}

\DeclareMathOperator{\tr}{tr}

\DeclareMathOperator{\ad}{ad}
\DeclareMathOperator{\Ad}{Ad}
\DeclareMathOperator{\Nij}{Nij}

\author{Cristian Ortiz and Carlos Varea}
\address{}

\address{Cristian Ortiz\newline
\indent Instituto de Matemática e Estatística\newline
\indent Universidade de São Paulo,\newline
\indent Rua do Matão, 1010 - Cidade Universitária,\newline
\indent 05508-090, São Paulo - São Paulo - Brasil.\newline
\indent cortiz@ime.usp.br}

\address{Carlos Varea\newline
\indent Departamento de Matemática\newline
\indent Universidade Tecnológica Federal do Paraná,\newline
\indent Av. Alberto Carazzai, 1640 - Centro,\newline
\indent 86300-000, Cornélio Procópio - Paraná - Brasil.\newline
\indent carlosvarea@utfpr.edu.br}

\title{Complex Dirac structures with constant real index on flag manifolds}

\begin{document}
\maketitle

\begin{abstract}
In this paper we describe all invariant complex Dirac structures with constant real index on a maximal flag manifold in terms of the roots of the Lie algebra which defines the flag manifold. We also completely classify these structures under the action of $B$-transformations. 
\end{abstract}

\tableofcontents

\section{Introduction}

A Dirac structure \cite{C} on a manifold $M$ is a subbundle $L \subset TM\oplus T^\ast M$ which is both maximal isotropic with respect to the natural symmetric pairing defined on $TM\oplus T^\ast M$ and involutive with respect to the Courant bracket. The natural pairing and the Courant bracket on $TM\oplus T^\ast M$ extend naturally to the complexification $(TM\oplus T^\ast M)\otimes \mathbb{C}$, allowing us to define complex Dirac structures. If  $L \subset (TM\oplus T^\ast M)\otimes \mathbb{C}$ is a complex Dirac structure and $x\in M$,  the \textbf{real index} of $L$ at $x$ is defined as $\dim (L_x \cap \overline{L_x})$. A generalized complex structure \cite{H, G1} on $M$ can be seen as a complex Dirac structure whose real index at every point is zero. Hence, generalized complex structures are particular instances of complex Dirac structures with constant real index. These were introduced in \cite{A,AR} including not only generalized complex structures but also presymplectic, transverse holomorphic and CR structures.

A diffeomorphism $\phi:M\to M$ lifts naturally to a vector bundle automorphism $\tilde{\phi}:(TM\oplus T^\ast M)\otimes \mathbb{C}\to (TM\oplus T^\ast M)\otimes \mathbb{C}$. One observes that if $\tilde{\phi}(L_x)=L_{\phi(x)}$, then the real index of $L$ at $x\in M$ coincides with the real index at $\phi(x)$. In particular, if $G$ is a Lie group acting on a manifold $M$ via $\phi_g:M\to M$ and $L\subseteq (TM\oplus T^\ast M)\otimes \mathbb{C}$ is an invariant complex Dirac structure, i.e.  $\tilde{\phi_g}(L_x)=L_{\phi_g(x)}$ for every $x\in M$ and $g\in G$, then the real index of $L$ is constant along the orbits of the action. As a consequence, every invariant complex Dirac structure on a homogeneous space has constant real index.

The study of geometric structures which are invariant by the action of a Lie group plays a central role in differential geometry, including: symplectic actions and moment maps, isometric actions, holomorphic actions, among others. Directly related to this work is the study of invariant geometric structures on homogeneous spaces such as K\"ahler structures \cite{Borel} and generalized complex structures on both nilmanifolds \cite{CG} and flag manifolds \cite{VS,V}.

In this paper we are concerned with invariant complex Dirac structures on maximal flag manifolds of complex semisimple Lie groups. Let $\mathfrak{g}$ be a complex semisimple Lie algebra and let $G$ be a connected Lie group with Lie algebra $\mathfrak{g}$. Then its maximal flag manifold is the homogeneous space $\mathbb{F} = G/P$ where $P$ is a Borel subgroup of $G$ (minimal parabolic subgroup). If $U$ is a compact real form of $G$, then $U$ acts transitively on $\mathbb{F}$ so that we also have the homogeneous space $\mathbb{F} = U/T$ where $T = P\cap U$ is a maximal torus of $U$. We are concerned with $U$-invariant structures on $\mathbb{F}$. In general, invariant structures on a flag manifold $\mathbb{F}$ are expected to be described in terms of roots of the Lie algebra which defines the flag manifold, see for instance \cite{BH1,BH2,SN}.

In \cite{VS}, the second author studied invariant generalized complex structures on $\mathbb{F}$. In this work, we show that the techniques of \cite{VS} naturally extend to the setting of invariant complex Dirac structures. The idea is to reduce the problem at the origin $b_0$ of $\mathbb{F}$, then an invariant complex Dirac structure on $\mathbb{F}$ will be completely described by a complex Dirac structure $L$ on the vector space $T_{b_0} \mathbb{F} \oplus T_{b_0} ^\ast \mathbb{F}$ which is invariant under the adjoint representation. We can decompose the Lie algebra of $U$ as $\mathfrak{u} = \mathfrak{t} \oplus \mathfrak{m}$ where $\mathfrak{t}$ is the Lie algebra of $T$ and $\mathfrak{m} = \sum _{\alpha} \mathfrak{u}_\alpha$ is the sum of root spaces in $\mathfrak{u}$, that is, $\mathfrak{u}_\alpha = (\mathfrak{g}_\alpha + \mathfrak{g}_{-\alpha})\cap \mathfrak{u}$ and $\mathfrak{g}_\alpha$ is the root space in the complex Lie algebra $\mathfrak{g}$. Thus we can identify $T_{b_0} \mathbb{F} \oplus T_{b_0} ^\ast \mathbb{F}$ with two copies of $\mathfrak{m}$, namely $T_{b_0} \mathbb{F} \oplus T_{b_0} ^\ast \mathbb{F} \approx \mathfrak{m}\oplus \mathfrak{m}^\ast$ where $\mathfrak{m}^\ast = \sum_\alpha \mathfrak{u}_\alpha ^\ast$. In the same spirit of \cite{VS}, we show Proposition \ref{rootdecomposition}, which says that every invariant complex Dirac structure on $\mathfrak{m}$ decomposes as a direct sum $L = \bigoplus _\alpha L_\alpha$ where each $L_\alpha$ is an invariant complex Dirac structure on $\mathfrak{u}_\alpha$.

Once we have an algebraic description of a complex Dirac structure, we proceed to analyze the involutivity with respect to the Courant bracket. We obtain algebraic conditions for a complex Dirac structure to be involutive depending on a triple of positive roots $(\alpha,\beta, \alpha+\beta)$. Thus, in order to simplify the computations we separate the cases according to the real index of the corresponding triple of subspaces $(L_\alpha,L_\beta,L_{\alpha+\beta})$.

As observed before, invariant complex Dirac structures on a homogeneous space always have constant real index. Our main results are the following.

\begin{theorem*}
Let $L$ be an invariant complex Dirac structure on a flag manifold $\mathbb{F}$. Then the real index of $L$ is constant and equal to $2k$ for some $0\leq k \leq l$ where $l$ is the number of positive roots of the associated Lie algebra $\mathfrak{g}$.  
\end{theorem*}

Conversely, if $l$ is the number of positive roots of the Lie algebra $\mathfrak{g}$ which define the flag manifold $\mathbb{F}$, then there exists an invariant complex Dirac structure on $\mathbb{F}$ with constant real index equal to $2k$, with $k \in \{0,1,\cdots ,l\}$.

\begin{theorem*}
Let $\mathfrak{g}$ be a semisimple Lie algebra and consider $\mathbb{F}$ the maximal flag manifold associated to $\mathfrak{g}$. If $\mathfrak{g}$ has $l$ positive roots, then there exists an invariant complex Dirac structure on $\mathbb{F}$ with constant real index equal to $2k$, where $0\leq k \leq l$.
\end{theorem*}

It is well-known that the vector bundle $(TM\oplus T^\ast M)\otimes \mathbb{C}$ has additional symmetries which do not come from a diffeomorphism of the base $M$, namely symmetries given by $B$-fields. Bringing $B$-field invariance into the picture allows us to fully describe all invariant complex Dirac structures up to $B$-fields. Using Proposition \ref{rootdecomposition}, which gives a decomposition of an invariant complex Dirac structure in terms of $\mathfrak{u}_\alpha$, one can show the following result.

\begin{proposition*}
Let $L_\alpha$ be an invariant complex Dirac structure on $\mathfrak{u}_\alpha$. Then, up to $B$-transformations, we have only four possibilities for $L_\alpha$.
\begin{itemize}
\item[(a)] $L_\alpha = (\mathfrak{u}_\alpha)_\mathbb{C}$, that is, $L_\alpha$ is the root space associated to the root $\alpha$;

\item[(b)]  $L_\alpha = (\mathfrak{u}^\ast _\alpha)_\mathbb{C}$, that is, $L_\alpha$ is the dual of the root space associated to the root $\alpha$;

\item[(c)] $L_\alpha = \textnormal{span}_\mathbb{C}\{ A_\alpha + \varepsilon_\alpha iS_\alpha , A_\alpha ^\ast + \varepsilon_\alpha i S^\ast _\alpha \}$ with $\varepsilon_\alpha = \pm 1$, that is, $L_\alpha$ is the $i$-eigenspace of an invariant generalized complex structure on $\mathfrak{u}_\alpha$ of complex type.

\item[(d)] $L_\alpha = \textnormal{span}_\mathbb{C}\{ A_\alpha + \frac{i}{x_\alpha} A_\alpha ^\ast, S_\alpha + \frac{i}{x_\alpha} S^\ast _\alpha \}$ with $x_\alpha,a_\alpha \in \mathbb{R}$ and $x_\alpha \neq 0$, that is, $L_\alpha$ is the $i$-eigenspace of an invariant generalized complex structure on $\mathfrak{u}_\alpha$ of symplectic type.
\end{itemize}
In other words, if we denote by $\textnormal{Dir}_\mathbb{C} (\mathfrak{u}_\alpha)$ the set of all invariant complex Dirac structures on $\mathfrak{u}_\alpha$, then 
\[
\frac{\textnormal{Dir}_\mathbb{C}(\mathfrak{u}_\alpha)}{B\textnormal{-fields}} = \{ (a),(b),(c),(d)\}.
\]
\end{proposition*}

The paper is organized as follows. In Section \ref{sectionpreliminar} we present a brief introduction to Dirac structures and, as a particular case, we define generalized complex structures. These are the basic concepts used throughout the paper. 

In Section \ref{sectionflag} we introduce the concept of flag manifolds and fix notation that will be useful in the paper.

In Section \ref{sectionDirac} we recall the classification of invariant generalized complex structures on a maximal flag manifold as in \cite{VS}. Since generalized complex structures are a special case of complex Dirac structures, we use this classification as a motivation for the study of general invariant complex Dirac structures on a maximal flag. Thus we give a description of the invariant complex Dirac structures on $\mathbb{F}$. To obtain such a  description we first describe the maximal isotropic subspaces of $\mathfrak{u}_\alpha$, for a fixed root $\alpha$. Then we analyze the invariance and the real index of them and, finally, we study the involutivity by considering different cases according to the real index. 





In Section \ref{sectionBtransformation} we study the action of $B$-transformations on the space of invariant complex Dirac structures found in Section \ref{sectionDirac}, presenting a classification of the possible invariant complex Dirac structures on flag manifolds up to $B$-transformations. 

Finally, Section \ref{secexample} is devoted to give explicit examples of invariant complex Dirac structures for the maximal flag manifolds associated to the Lie algebras $\mathfrak{sl}(2,\mathbb{C})$ and $\mathfrak{sl}(3,\mathbb{C})$.\\

\textbf{Acknowledgments:} We would like to thank Dan Aguero and Fabricio Valencia for useful conversations. We are specially thankful to Elizabeth Gasparim and Roberto Rubio for several suggestions and comments that have improved this manuscript. C. Ortiz was supported by Grant 2016/23746-6 Sao Paulo Research Foundation - FAPESP  and by the National Council of Research and Development CNPq-Brazil, Bolsa de Produtividade em Pesquisa Grant 315502/2020-7. C. Varea was supported by the Sao Paulo Research Foundation - FAPESP Grant 2020/12018-5.


\section{Dirac structures and generalized complex structures}\label{sectionpreliminar}

In this section we introduce the basics on Dirac structures and generalized complex structures. For more details see \cite{Bu}, \cite{C} and \cite{G2}.


Let $M$ be a smooth $n$-dimensional manifold. The vector bundle $TM\oplus T^{\ast }M$ is endowed with a natural symmetric bilinear
form of signature $(n,n)$ defined by 
\begin{equation*}
\langle X+\xi ,Y+\eta \rangle =\frac{1}{2}(\xi (Y)+\eta (X)),
\end{equation*}%
where $X+\xi,Y+\eta \in \Gamma (TM\oplus T^\ast M)$. The Courant bracket is the skew-symmetric bracket defined on $\Gamma(TM\oplus T^{\ast }M)$ by 
\begin{equation*}
\lbrack X+\xi ,Y+\eta ]=[X,Y]+\left( \mathcal{L}_{X}\eta -\mathcal{L}_{Y}\xi -\frac{%
1}{2}d\left( i_{X}\eta -i_{Y}\xi \right)\right) .
\end{equation*}

\begin{definition}(\cite{C})
A {\bf Dirac structure} on $M$ is a maximal isotropic subbundle $L\subseteq TM\oplus T^\ast M$, whose space of sections is closed under the Courant bracket. 
\end{definition}

\begin{example} Examples of Dirac structures are:
\begin{itemize}
\item[a)] Presymplectic structures: let $\omega \in \Omega (M)$ be a nondegenerate closed $2$-form. Then
\[
L_\omega = \{ X+\iota_X \omega \ | \ X\in TM\}
\]
is a Dirac structure.

\item[b)] Poisson structures: let $\pi$ be a Poisson structure. Then 
\[
\textnormal{Graph}(\pi) = \{ \iota_\xi \pi+ \xi \ | \ \xi \in T^\ast M\}
\]
is a Dirac structure.
\end{itemize}
\end{example}

It is important to notice that these notions can be extended to the complexification $(TM\oplus T^\ast M)\otimes \mathbb{C}$ which allows us to define a complex Dirac structure on $M$.

\begin{definition}
A {\bf complex Dirac structure} on $M$ is an involutive maximal isotropic subbundle $L\subseteq (TM\oplus T^\ast M)\otimes \mathbb{C}$.
\end{definition}

Given $L$ a complex Dirac structure on $M$ and $p\in M$, we note that $L|_p$ is a complex Dirac structure on $T_p M$.  

\begin{definition}\label{realindex}
The {\bf real index} of a complex Dirac structure $L|_p \subset (T_p M\oplus T_p ^\ast M)\otimes \mathbb{C}$ is given by $\dim_\mathbb{C} (L|_p \cap \overline{L|}_p)$.
\end{definition}

A very special class of complex Dirac structures is given by those with real index zero, called {\bf generalized complex structures}. The condition $L\cap \overline{L} = \{0 \}$ allows us to describe a generalized complex structure as an endomorphism $\mathcal{J}\colon TM \oplus T^\ast M \rightarrow TM \oplus T^\ast M$ satisfying $\mathcal{J}^2 = -1$ and $\mathcal{J}^\ast = -\mathcal{J}$. 

The basic examples of generalized complex structures on a manifold $M$ are: 
\begin{equation*}
\mathcal{J}_{J}=\left( 
\begin{array}{cc}
-J & 0 \\ 
0 & J^{\ast }%
\end{array}%
\right) \ \mathrm{and}\ \mathcal{J}_{\omega }=\left( 
\begin{array}{cc}
0 & -\omega ^{-1} \\ 
\omega & 0%
\end{array}%
\right)
\end{equation*}%

\noindent where $J:TM\to TM$ is a complex structure and $\omega\in\Omega^2(M)$ is a symplectic form.

Note that, given $L$ a maximal isotropic subbundle of $TM\oplus T^{\ast }M$
(or its complexification) then $L$ is Courant involutive if and only if $\Nij%
|_{L}=0$, where  
\begin{equation*}
\Nij(A,B,C)=\frac{1}{3}\left( \langle \lbrack A,B],C\rangle +\langle \lbrack
B,C],A\rangle +\langle \lbrack C,A],B\rangle \right),
\end{equation*}

\noindent is the Nijenhuis operator.

\section{Flag manifolds}\label{sectionflag}

The aim of this section is to briefly review the concept of a flag manifold. Let $\mathfrak{g}$ be a semisimple Lie algebra and $G$ be a connected Lie group with Lie algebra $\mathfrak{g}$. Let us consider a Cartan subalgebra $\mathfrak{h}\subseteq \mathfrak{g}$ and $\Pi$ be a root system of $\mathfrak{g}$ relative to $\mathfrak{h}$. For each root $\alpha \in \Pi$, denote by $\mathfrak{g}_\alpha = \{ X\in \mathfrak{g} \ | \ \ad_\mathfrak{h} (X) = \alpha(\mathfrak{h})\cdot X \}$ the root space associated to $\alpha$. The Cartan--Killing form of $\mathfrak{g}$ is defined by $\langle X,Y \rangle = \tr (\ad_X \cdot \ad_Y )$ and it restricts to a nondegenerate bilinear form on $\mathfrak{h}$. Given $\alpha \in \mathfrak{h}^\ast$ we denote by $H_\alpha \in \mathfrak{h}$ the element which is defined by $\alpha ( \cdot ) = \langle H_\alpha , \cdot \rangle$, and denote by $\mathfrak{h}_\mathbb{R} \subseteq \mathfrak{h}$ the real subspace generated by $H_\alpha$, $\alpha \in \Pi$.

Let us choose a Weyl basis of $\mathfrak{g}$. Such a basis is given by root vectors  $X_\alpha \in \mathfrak{g}_\alpha$ satisfying $\langle X_\alpha,X_{-\alpha} \rangle = 1$ and 
$$[X_\alpha,X_\beta ] = \left\lbrace \begin{array}{l}
m_{\alpha,\beta} X_{\alpha+\beta}, \ \textnormal{ if } \alpha +\beta \ \textnormal{ is a root} \\
0, \ \textnormal{ otherwise}
\end{array}\right.$$ with $m_{\alpha,\beta} \in \mathbb{R}$ obeying $m_{-\alpha,-\beta} = -m_{\alpha,\beta}$.

Let $\Pi^+ \subset \Pi$ be a choice of positive roots and denote by $\Sigma$ the corresponding simple root system. We define the {\bf Borel subalgebra} $\mathfrak{b}$ of $\mathfrak{g}$ by
\[
\mathfrak{b} = \mathfrak{h} \oplus \sum_{\alpha \in \Pi^+} \mathfrak{g}_\alpha.
\]
A subalgebra $\mathfrak{p}$ of $\mathfrak{g}$ is {\bf parabolic} if $\mathfrak{b} \subseteq \mathfrak{p}$. In particular, when $\mathfrak{p} = \mathfrak{b}$ we say that the parabolic subalgebra is {\bf minimal}.


\begin{definition}
Let $\mathfrak{p}$ be a minimal parabolic subalgebra. The {\bf maximal flag manifold} $\mathbb{F}$ associated to $\mathfrak{p}$ is the homogeneous space $\mathbb{F} = G/P$, where $P\subseteq G$ is the parabolic subgroup generated by $\mathfrak{p}$. 
\end{definition}


We also have a realization of a flag manifold using a compact real form of $G$. Indeed, let $\mathfrak{u}$ be a compact real form of $\mathfrak{g}$, that is, the real subalgebra 
\[
\mathfrak{u} = \textnormal{span}_\mathbb{R} \{ i\mathfrak{h}_\mathbb{R},A_\alpha, S_\alpha \ : \ \alpha\in \Pi^+ \},
\]
where $A_\alpha = X_\alpha - X_{-\alpha}$ and $S_\alpha = i(X_\alpha + X_{-\alpha})$. Denote by $U = \exp \mathfrak{u}$ the corresponding compact real form of $G$. Then the real representation $\mathbb{F} = U/T$ is obtained by the transitive action of $U$ on $G/P$, where the closed connected subgroup $T = P \cap U$ can be identified with a maximal torus on $U$. The Lie algebra of $T$ is $\mathfrak{t} = \mathfrak{p} \cap \mathfrak{u}$.  Let $\mathfrak{m}$ be the orthogonal complement (with respect to the Killing form) of $\mathfrak{t}$ in $\mathfrak{u}$. Since $\mathbb{F}$ is a reductive homogeneous space, its tangent space at the origin can be identified with $\mathfrak{m}$, that is, $T_{b_0} \mathbb{F} \simeq \mathfrak{m}$ where $b_0 = e\cdot T$ is the origin of the flag manifold. 
Throughout the text we will consider only maximal flag manifolds, so when we consider $\mathbb{F}$ a flag manifold it is understood that $\mathbb{F}$ is a maximal flag manifold.

\section{Complex Dirac structures on flag manifolds}\label{sectionDirac}

Our goal in this section is to describe the invariant complex Dirac structures on a maximal flag manifold $\mathbb{F} = U/T$. 

\begin{definition}\label{definvariance}
Let $L$ be a complex Dirac structure on $\mathbb{F}$. We say that $L$ is {\bf invariant} if $L$ is invariant under action of the adjoint representation. That is, 

$$(\Ad (g) \oplus \Ad ^\ast (g) )\cdot L \subseteq L,$$

\noindent for every $g \in T$.
\end{definition}

\begin{example}[Invariant generalized complex structures on $\mathbb{F}$]

In \cite{VS} was proved that an invariant generalized almost complex structure $\mathcal{J}$ on $\mathbb{F}$ can be decomposed as $\mathcal{J} = \bigoplus_{\alpha} \mathcal{J}_\alpha$ where $\mathcal{J}_\alpha$ is the restriction of $\mathcal{J}$ to the subspace $\mathfrak{u}_\alpha \oplus \mathfrak{u}_\alpha ^\ast$, where $\mathfrak{u}_\alpha = \textnormal{span}_\mathbb{R}\{A_\alpha, S_\alpha \}$, and each $\mathcal{J}_\alpha$ can assume only two forms:
\begin{itemize}
\item[a)] Complex type
\[
\mathcal{J}_\alpha = \pm \mathcal{J}_0=\pm \left( 
\begin{array}{cccc}
0 & -1 & 0 & 0\\ 
1 & 0 & 0 & 0\\
0 & 0 & 0 & -1\\
0 & 0 & 1 & 0
\end{array}%
\right).
\]

\item[b)] Symplectic type (up to $B$-transformation)
\[
\mathcal{J}_\alpha=\left( 
\begin{array}{cccc}
a_\alpha & 0 & 0 & -x_\alpha\\ 
0 &  a_\alpha& x_\alpha & 0\\
0 & -y_\alpha & -a_\alpha & 0\\
y_\alpha & 0 & 0 & -a_\alpha
\end{array}%
\right)
\]

\noindent with $a_\alpha,x_\alpha,y_\alpha\in\mathbb{R}$ such that $a_\alpha^2=x_\alpha y_\alpha-1$.
\end{itemize}
The integrability of the invariant generalized almost complex structures described above depends on analyzing what happens with triples of the form $(\mathcal{J}_\alpha,\mathcal{J}_\beta,\mathcal{J}_{\alpha+\beta})$
associated to triples of positive roots $(\alpha,\beta,\alpha+\beta)$. 
Accordingly, we obtain that $\mathcal{J}$ is integrable if and only if for each triple of positive roots $(\alpha,\beta,\alpha+\beta)$ we have that $(\mathcal{J}_\alpha,\mathcal{J}_\beta,\mathcal{J}_{\alpha+\beta})$ 
corresponds to one of the rows of the following table

\begin{table}[htb!]
\begin{center}
\begin{tabular}{|c|c|c|}
\hline $\mathcal{J}_\alpha$ & $\mathcal{J}_\beta$ & $\mathcal{J}_{\alpha + \beta}$ \\
\hline
complex & complex & complex \\
symplectic  & complex ($\pm \mathcal{J}_0 $) & complex ($\pm \mathcal{J}_0)$ \\
complex ($\pm \mathcal{J}_0$) & symplectic & complex ($\pm \mathcal{J}_0$) \\
complex ($\pm \mathcal{J}_0)$ & complex ($\mp \mathcal{J}_0)$ & symplectic\\
symplectic & symplectic & symplectic \\ \hline
\end{tabular} 
\caption{Integrability conditions.}
\label{Jintegravel}
\end{center}
\end{table}
\FloatBarrier
 
\noindent where, in the first row, we can have one of the following combination of signs $(\pm \mathcal{J}_0, \pm \mathcal{J}_0, \pm \mathcal{J}_0)$, $(\pm \mathcal{J}_0, \mp \mathcal{J}_0, \pm \mathcal{J}_0)$ and $(\pm \mathcal{J}_0, \mp \mathcal{J}_0, \mp \mathcal{J}_0)$. Moreover, in the last row, the following extra condition is required for integrability:
\begin{equation}\label{Integrabilityconditions}
\left\lbrace \begin{array}{cc}
a_{\alpha+\beta} x_\alpha x_\beta - a_\beta x_\alpha x_{\alpha+\beta} - a_\alpha x_\beta x_{\alpha+\beta} = 0 \\
x_\alpha x_\beta - x_\alpha x_{\alpha+\beta} - x_\beta x_{\alpha+\beta} = 0.
\end{array}\right.
\end{equation} 
For more details see \cite{VS}.
\end{example}

Our main goal is to generalize the results from invariant generalized complex structures \cite{VS} to the setting of complex Dirac structures with constant real index. To this end, we start by recalling the following result from \cite{VS}.

\begin{proposition}\label{rootdecomposition}
Let $L$ be a subspace of $\mathfrak{m}\oplus \mathfrak{m}^\ast$. There is a natural decomposition 
\[
L=\sum _{\alpha} L\cap (\mathfrak{u}_\alpha \oplus \mathfrak{u}_\alpha ^\ast ).
\]
Moreover, $L$ is a Dirac structure on $\mathfrak{m}$ if and only if $L_\alpha := L\cap (\mathfrak{u}_\alpha \oplus \mathfrak{u}_\alpha ^\ast)$ is a Dirac structure on $\mathfrak{u}_\alpha$ for each positive root $\alpha$. 
\end{proposition}
  
Note that Proposition \ref{rootdecomposition} can be extended to complex Dirac structures. So, for our purposes it suffices to describe the invariant complex Dirac structures $L_{\alpha}\subseteq \mathfrak{u}_\alpha\oplus \mathfrak{u}^*_\alpha$ for each positive root $\alpha$. For that, first we will find all maximal isotropic subspaces of $(\mathfrak{u}_\alpha \oplus \mathfrak{u}^\ast _\alpha)\otimes \mathbb{C}$, then we proceed to analyze the invariance, involutivity and real index of such subspaces separately.  

\begin{remark}
It was proven in \cite{AR} that every complex Dirac structure has an associated canonical generalized complex structure (see \cite{AR}, Proposition 4.2). Then it is natural to expect that this result combined with \cite{VS} would provide a classification of the invariant complex Dirac structures on a flag manifold. However, the generalized complex structure of (\cite{AR}, Proposition 4.2) is not necessarily defined on a flag manifold as well. Because of this we focus our study on invariant complex Dirac structures rather than invariant generalized complex structures.
\end{remark}


\subsection{Maximal isotropic subspaces of $(\mathfrak{u}_\alpha \oplus \mathfrak{u}^\ast _\alpha)\otimes \mathbb{C}$}
Here we describe all maximal isotropic subspaces of $(\mathfrak{u}_\alpha \oplus \mathfrak{u}^\ast _\alpha)\otimes \mathbb{C}$.

\begin{proposition}\label{maximalisotropicsubspaces}
Let $L = \sum_\alpha L_\alpha$ be a maximal isotropic subspace of $(\mathfrak{m}\oplus \mathfrak{m}^\ast)\otimes \mathbb{C}$. Then, for each root $\alpha\in \Pi$, $L_\alpha\subseteq \mathfrak{u}_{\alpha}\oplus \mathfrak{u}^*_{\alpha}$ must be one of the following:
\begin{itemize}
\item[(a)] $L_\alpha = \textnormal{span}_\mathbb{C} \{ A_\alpha,S_\alpha\} = (\mathfrak{u}_\alpha)_\mathbb{C}$;
\item[(b)] $L_\alpha = \textnormal{span}_\mathbb{C} \{ -S_\alpha ^\ast ,A_\alpha ^\ast \} = (\mathfrak{u}_\alpha ^\ast)_\mathbb{C}$;
\item[(c)] $L_\alpha = \textnormal{span}_\mathbb{C} \{ A_\alpha,A_\alpha ^\ast \}$;
\item[(d)] $L_\alpha = \textnormal{span}_\mathbb{C} \{ S_\alpha,-S_\alpha ^\ast \}$;
\item[(e)] $L_\alpha = \textnormal{span}_\mathbb{C} \{ a_\alpha A_\alpha+b_\alpha S_\alpha, a_\alpha A_\alpha ^\ast + b_\alpha S_\alpha ^\ast \}$ with $a_\alpha,b_\alpha \neq 0$;
\item[(f)] $L_\alpha = \textnormal{span}_\mathbb{C} \{ a_\alpha A_\alpha+b_\alpha A_\alpha ^\ast,a_\alpha	S_\alpha + b_\alpha S_\alpha ^\ast \}$ with $a_\alpha,b_\alpha \neq 0$.
\end{itemize}
\end{proposition}
\begin{proof}
Using that $\langle A_\alpha, S_\alpha ^\ast\rangle = - \langle S_\alpha, A^\ast _\alpha\rangle = 1$ and $\langle A_\alpha, A_\alpha ^\ast\rangle = \langle A_\alpha, S_\alpha \rangle = \langle A_\alpha, A_\alpha ^\ast\rangle = \langle S_\alpha, S_\alpha \rangle = \langle S_\alpha, S_\alpha ^\ast\rangle = \langle A_\alpha ^\ast, A_\alpha ^\ast\rangle = \langle A_\alpha ^\ast, S_\alpha ^\ast\rangle = \langle S_\alpha ^\ast, S_\alpha ^\ast\rangle = 0$ we can check directly that the subspaces $(a)$ -- $(f)$ are in fact maximal isotropic subspaces on $(\mathfrak{m}\oplus \mathfrak{m}^\ast)\otimes \mathbb{C}$.

Now let us prove that these are the only possibilities for a maximal isotropic subspace on $(\mathfrak{u}_\alpha \oplus \mathfrak{u}_\alpha ^\ast) \otimes \mathbb{C}$. In fact, let $L_\alpha$ be a maximal isotropic subspace on $(\mathfrak{u}_\alpha \oplus \mathfrak{u}_\alpha ^\ast) \otimes \mathbb{C}$. Since $(\mathfrak{u}_\alpha \oplus \mathfrak{u}_\alpha ^\ast) \otimes \mathbb{C}$ is 4-dimensional, it follows that $L_\alpha$ is 2-dimensional. Indeed, $L_\alpha = \textnormal{span}_\mathbb{C} \{ X,Y\}$ where $X=a_1A_\alpha +b_1S_\alpha+c_1 (-s_\alpha ^\ast)+d_1 A_\alpha ^\ast$ and $Y=a_2A_\alpha +b_2S_\alpha+c_2 (-s_\alpha ^\ast)+d_2 A_\alpha ^\ast$, with $a_i,b_i,c_i,d_i\in \mathbb{C}$. Therefore, we have $\langle X,X \rangle = 2(a_1c_1 + b_1d_1)$, $\langle X,Y \rangle = a_1c_2+b_1d_2+c_1a_2+d_1b_2$ and $\langle Y,Y\rangle = 2(a_2c_2+b_2d_2)$. Now, using these expressions and the fact that $L_\alpha$ is isotropic, we can prove that $L_\alpha$ must be one of those subspaces $(a)$ -- $(f)$. The idea is to make some of the scalars $a_1,b_1,c_1,d_1$ equal to zero and analyze each case until the possibilities are exhausted. Thus we have 14 possibilities to analyze. We will do the first case and the other cases are similar.

Consider $b_1=c_1=d_1 = 0$, that is, $X= a_1 \alpha$ with $a_1 \neq 0$. Then we have that $\langle X,Y \rangle= 0$ if and only if $a_1c_2 = 0$, therefore $c_2 = 0$. Then we can write $Y= a_2 A_\alpha + b_2 S_\alpha + d_2 A^\ast _\alpha$. Moreover, since $\langle Y,Y \rangle= 0$ we obtain $b_2d_2 = 0$. So, we get two subcases:
\begin{itemize}
\item[(i)] If $c_2 = b_2 = 0$, we have $Y= a_2 A_\alpha + d_2 A^\ast _\alpha$. Thus,
\begin{eqnarray*}
L_\alpha & = & \textnormal{span}_\mathbb{C}\{a_1 A_\alpha , a_2A_\alpha +d_2A^\ast _\alpha \}\\
& = & \textnormal{span}_\mathbb{C}\{ A_\alpha, A^\ast _\alpha\}.
\end{eqnarray*}

\item[(ii)] If $c_2 = d_2 = 0$, then $Y= a_2 A_\alpha + b_2S_\alpha$. Hence,
\begin{eqnarray*}
L_\alpha & = & \textnormal{span}_\mathbb{C}\{a_1 A_\alpha , a_2A_\alpha +b_2S_\alpha \}\\
& = & \textnormal{span}_\mathbb{C}\{ A_\alpha, S_\alpha\}.
\end{eqnarray*}
\end{itemize}
Therefore, we conclude that subcases $(i)$ and $(ii)$ are exactly the cases $(c)$ and $(a)$, respectively.
\end{proof}

\subsection{Invariance}
  
Remember that we are interested in the subspaces which are invariant with respect to the isotropy representation. In our case, the isotropy representation is equivalent to the adjoint representation which can be identified with the torus action as we saw in \cite{VS}. So these subspaces $L_\alpha$ listed in Proposition \ref{maximalisotropicsubspaces} are invariant if
\[
\left(\begin{array}{cccc}
\cos \theta & -\sin \theta & 0 & 0 \\
\sin \theta & \cos \theta & 0 & 0\\
0 & 0 & \cos \theta & -\sin \theta \\
0 & 0 & \sin \theta & \cos \theta
\end{array}\right) \cdot L_\alpha \subset L_\alpha.
\]
It is simple to see that the subspaces in $(a)$ and $(b)$ in Proposition \ref{maximalisotropicsubspaces} are invariant. Having into account that we are considering the ordered basis $\mathcal{B} = \{A_\alpha,S_\alpha,-S_\alpha ^\ast,A_\alpha ^\ast \}$, let us check the other cases:
\begin{itemize}
\item[(c)] We have that $A_\alpha \in L_\alpha$, then
\[
\left(\begin{array}{cccc}
\cos \theta & -\sin \theta & 0 & 0 \\
\sin \theta & \cos \theta & 0 & 0\\
0 & 0 & \cos \theta & -\sin \theta \\
0 & 0 & \sin \theta & \cos \theta
\end{array}\right) \cdot \left( \begin{array}{c} 
1 \\ 0 \\ 0 \\ 0
\end{array}\right) = \left( \begin{array}{c} 
\cos \theta \\ \sin \theta \\ 0 \\ 0
\end{array}\right).
\]
But observe that $\cos \theta A_\alpha + \sin \theta S_\alpha \notin L_\alpha$ when $\sin \theta \not= 0$. Therefore, we conclude that $L_\alpha$ is not invariant.

\item[(d)] The reasoning is analogous to case $(c)$ using the fact that $S_\alpha \in L_\alpha$, concluding that $L_\alpha$ is not invariant.

\item[(e)] Take $a_1 A_\alpha +b_1 S_\alpha \in L_\alpha$, then 
\[
\left(\begin{array}{cccc}
\cos \theta & -\sin \theta & 0 & 0 \\
\sin \theta & \cos \theta & 0 & 0\\
0 & 0 & \cos \theta & -\sin \theta \\
0 & 0 & \sin \theta & \cos \theta
\end{array}\right) \cdot \left( \begin{array}{c} 
a_1 \\ b_1 \\ 0 \\ 0
\end{array}\right) = \left( \begin{array}{c} 
a_1 \cos \theta - b_1 \sin \theta \\ a_1 \sin \theta + b_1 \cos \theta \\ 0 \\ 0
\end{array}\right).
\]
But $(a_1 \cos \theta - b_1 \sin \theta)A_\alpha + (a_1 \sin \theta + b_1 \cos \theta)S_\alpha \in L_\alpha$ if and only if there exists a scalar $k_1$ such that $(a_1 \cos \theta - b_1 \sin \theta)A_\alpha + (a_1 \sin \theta + b_1 \cos \theta)S_\alpha = k_1 (a_1 A_\alpha +b_1 S_\alpha)$, in other words, the scalar $k$ satisfies
\begin{eqnarray}\label{sistcase5}
\left\lbrace \begin{array}{l}
a_1 \cos \theta - b_1 \sin \theta = k_1a_1\\
a_1 \sin \theta + b_1 \cos \theta = k_1b_1.
\end{array}\right. 
\end{eqnarray}
Observe that, for $\theta = \pi/2$ we have 
\[
\left\lbrace \begin{array}{l}
 - b_1 = k_1a_1\\
a_1 = k_1b_1
\end{array}\right.
\]
and, then we obtain $k_1 ^2 = -1$ which implies that $k_1 = \pm i$. Using the fact that $a_1A_\alpha ^\ast +b_1S_\alpha ^\ast \in L_\alpha$ we obtain an expression equivalent to Equation (\ref{sistcase5}). Therefore, we have $L_\alpha$ invariant when 
\[
L_\alpha = \textnormal{span}_\mathbb{C}\{ A_\alpha \pm iS_\alpha ,A_\alpha ^\ast \pm iS_\alpha ^\ast\}.
\]

\item[(f)] Using $a_1 A_\alpha + b_1 A^\ast _\alpha \in L_\alpha$, we obtain
\[
\left(\begin{array}{cccc}
\cos \theta & -\sin \theta & 0 & 0 \\
\sin \theta & \cos \theta & 0 & 0\\
0 & 0 & \cos \theta & -\sin \theta \\
0 & 0 & \sin \theta & \cos \theta
\end{array}\right) \cdot \left( \begin{array}{c} 
a_1 \\ 0 \\ 0 \\ b_1
\end{array}\right) = \left( \begin{array}{c} 
a_1 \cos \theta\\ a_1 \sin \theta \\ -b_1 \sin \theta \\ a_1 \cos \theta
\end{array}\right).
\]
and then $a_1 \cos \theta A_\alpha + a_1 \sin \theta S_\alpha - b_1 \sin \theta (- S_\alpha ^\ast) + a_1 \cos \theta A_\alpha ^\ast \in L_\alpha$ if and only if there exist scalars $k_1,k_2$ such that $a_1 \cos \theta A_\alpha + a_1 \sin \theta S_\alpha - b_1 \sin \theta (- S_\alpha ^\ast) + a_1 \cos \theta A_\alpha ^\ast = k_1 (a_1 A_\alpha + b_1 A^\ast _\alpha) + k_2 (a_1 S_\alpha +b_1 S_\alpha ^\ast)$, or equivalently
\begin{eqnarray}\label{sistcase6}
\left\lbrace \begin{array}{l}
a_1 \cos \theta = k_1a_1\\
a_1 \sin \theta  = k_2a_1\\
b_1 \sin \theta = k_2b_1\\
b_1 \cos \theta = k_1b_1
\end{array}\right.
\end{eqnarray}
Since $a_1,b_1 \neq 0$ we have $k_1 = \cos \theta$ and $k_2 = \sin \theta$. Since $a_1 S_\alpha + b_1S^\ast _\alpha \in L_\alpha$ we obtain an expression equivalent to Equation (\ref{sistcase6}). Therefore, $L_\alpha$ is invariant.
\end{itemize}

Thus, we have the following result.

\begin{proposition}\label{listaDirac}
The invariant maximal isotropic subspaces of $(\mathfrak{u}_\alpha \oplus \mathfrak{u}_\alpha ^\ast)\otimes \mathbb{C}$ are:
\begin{enumerate}
\item[(A)] $L_\alpha = \textnormal{span}_\mathbb{C}\{ A_\alpha, S_\alpha\}$;

\item[(B)] $L_\alpha = \textnormal{span}_\mathbb{C}\{ -S^\ast _\alpha, A^\ast _\alpha\}$;

\item[(C)] $L_\alpha = \textnormal{span}_\mathbb{C}\{ A_\alpha+i\varepsilon_\alpha S_\alpha, A^\ast _\alpha +i\varepsilon_\alpha S_\alpha ^\ast\}$ with $\varepsilon_\alpha = \pm1$;

\item[(D)] $L_\alpha = \textnormal{span}_\mathbb{C}\{ a_1A_\alpha+ b_1A^\ast _\alpha, a_1S_\alpha + b_1S^\ast _\alpha\}$ with $a_1,b_1 \neq 0$.
\end{enumerate}
\end{proposition}

\subsection{Involutivity}\label{involutivity}

Now we study the Courant involutivity of a maximal isotropic subspace of $(\mathfrak{m}\oplus \mathfrak{m}^\ast)\otimes \mathbb{C}$ . In \cite{VS} was proved that the involutivity is given by analyzing the maximal isotropic subspaces associated to triples of positive roots of the type $(\alpha,\beta,\alpha+\beta)$, because for triples $(\alpha,\beta,\gamma)$, where $\gamma \neq \alpha +\beta$, the Nijenhuis operator is immediately zero. Since a maximal isotropic subspace of $(\mathfrak{m}\oplus \mathfrak{m}^\ast)\otimes \mathbb{C}$ can be written as $L=\sum_\alpha L_\alpha$, we just need to analyze the possible combinations for triples $(L_\alpha,L_\beta,L_{\alpha+\beta})$, where $L_\delta$ can be one of the four types listed on Proposition \ref{listaDirac}, for $\delta \in \{\alpha,\beta,\alpha+\beta\}$. 

Let $\alpha ,\beta $ be roots such that $\alpha +\beta $ is a root as well. After some computations, we have 
\begin{eqnarray*}
\Nij(A_{\alpha },S_{\beta },A_{\alpha +\beta }^{\ast }) & = & -\Nij(A_{\alpha
},A_{\beta },S_{\alpha +\beta }^{\ast })=\Nij(S_{\alpha },S_{\beta
},S_{\alpha +\beta }^{\ast })=\Nij(S_{\alpha },A_{\beta },A_{\alpha +\beta }^{\ast }) \\
& = & -\Nij (A_\alpha, S^\ast _\beta, A_{\alpha+\beta})=\Nij (A_\alpha,
A^\ast _\beta, S_{\alpha+\beta}) = -\Nij (S_\alpha, S^\ast _\beta, S_{\alpha+\beta})\\
& = & = - \Nij (S_\alpha, A^\ast _\beta, A_{\alpha+\beta}) = -\Nij(A_{\alpha }^{\ast },S_{\beta },A_{\alpha +\beta }) = \Nij(A_{\alpha
}^{\ast },A_{\beta },S_{\alpha +\beta })\\
& = & - \Nij(S_{\alpha }^{\ast },S_{\beta},S_{\alpha +\beta }) = -\Nij(S_{\alpha }^{\ast },A_{\beta },A_{\alpha +\beta }) \neq 0
\end{eqnarray*}
and the other cases are all zero. For more details see \cite{VS}.

From the last we can conclude the following.

\begin{proposition}\label{3tangent}
Let $(L_\alpha,L_\beta,L_{\alpha+\beta})$ be a triple where each of its components is described by case (A) of Proposition \ref{listaDirac}, then the Nijenhuis operator restricted to $L_\alpha\cup L_\beta \cup L_{\alpha+\beta}$ vanishes.
\end{proposition}

\begin{proposition}\label{2dualspaces}
Let $(L_\alpha,L_\beta,L_{\alpha+\beta})$ such that at least two of them are described by case (B) of Proposition \ref{listaDirac}, then the Nijenhuis operator restricted to $L_\alpha\cup L_\beta \cup L_{\alpha+\beta}$ vanishes. 
\end{proposition}

In what follows we will use the following notation:

\begin{notation}\label{nottriple}
We will refer to $L_\alpha$ by its respective case number according to Proposition \ref{listaDirac}. For example, let $(\alpha,\beta,\alpha+\beta)$  be a triple of roots and consider $(L_\alpha,L_\beta,L_{\alpha+\beta})$ the complex Dirac structures associated to these roots such that $L_\alpha = \textnormal{span}_\mathbb{C}\{ A_\alpha, S_\alpha\}$, $L_\beta = \textnormal{span}_\mathbb{C}\{ A_\beta+i\varepsilon_\beta S_\beta, A^\ast _\beta +i\varepsilon_\beta S_\beta ^\ast\}$ and $L_{\alpha+\beta} = \textnormal{span}_\mathbb{C}\{ a_1 A_{\alpha+\beta}+ b_1 A^\ast _{\alpha+\beta}, a_1 S_{\alpha+\beta} + b_1 S^\ast _{\alpha+\beta}\}$. Then we will refer to the triple $(L_\alpha,L_\beta,L_{\alpha+\beta})$ as $(A,C,D)$, because $L_\alpha$ is associated to case (A), $L_\beta$ to case (C) and $L_{\alpha+\beta}$ to case (D).
\end{notation}

We will list below the other cases where the Nijenhuis operator is zero. Fixing $L = L_\alpha \cup L_\beta \cup L_{\alpha+\beta}$, we have

\begin{table}[htb!]
\centering
\begin{tabular}{c|c|c|l}
$L_\alpha$ & $L_\beta$ & $L_{\alpha + \beta}$ & $\Nij|_L = 0$ \\ \hline
A & C & D & $\varepsilon_\beta = \varepsilon_{\alpha+\beta}$\\
A & D & D & $\frac{a_1}{b_1} = \frac{a_2}{b_2}$\\
B & C & D & $\varepsilon_\beta = \varepsilon_{\alpha+\beta}$\\
C & C & A & $\varepsilon_\alpha = -\varepsilon_\beta$\\
C & C & B & $\varepsilon_\alpha = -\varepsilon_\beta$\\
C & C & C & $\varepsilon_{\alpha+\beta}-\varepsilon_\alpha -\varepsilon_\beta +\varepsilon_\alpha \varepsilon_\beta \varepsilon_{\alpha+\beta} = 0$\\
C & C & D & $\varepsilon_\alpha = -\varepsilon_\beta$\\
C & D & C & $\varepsilon_\alpha = \varepsilon_{\alpha+\beta}$\\
D & D & A & $ \frac{a_1}{b_1} = -\frac{a_2}{b_2}$\\
D & D & D & $a_1a_2b_3-a_1b_2a_3-b_1a_2a_3=0$
\end{tabular}
\caption{Involutivity conditions} \label{tableinv}
\end{table}
\FloatBarrier

\subsection{The real index}
The next step is to determine the real index (see Definition \ref{realindex}) of each Lagrangian subspace of $(\mathfrak{u}_\alpha \oplus \mathfrak{u}_\alpha ^\ast)\otimes \mathbb{C}$. For the first two cases of Proposition \ref{listaDirac} it is simple to check that $L_\alpha = \overline{L_\alpha}$ and therefore $\dim (L_\alpha \cap \overline{L_\alpha}) = 2$. Moreover, note that the subspaces of type (C) in Proposition \ref{listaDirac} are the Lagrangian subspaces described in \cite{VS} associated to generalized complex structures of complex type, therefore $\dim (L_\alpha \cap \overline{L_\alpha}) = 0$. Then it remains case (D) of Proposition \ref{listaDirac} to be analyzed. We will see that in this case the real index of $L_{\alpha}$ can be either zero or two.

Indeed, consider $L_\alpha = \textnormal{span}_\mathbb{C}\{ a_1A_\alpha+b_1A_\alpha ^\ast, a_1S_\alpha + b_1S^\ast _\alpha\}$ with $a_1,b_1 \neq 0$ and suppose $X\in L_\alpha \cap \overline{L_\alpha}$. Then $X =  x_1 (a_1A_\alpha+b_1A_\alpha ^\ast) + x_2(a_1S_\alpha + b_1S^\ast _\alpha)$ and also $X = y_1 (\overline{a_1}A_\alpha+\overline{b_1}A_\alpha ^\ast) + y_2(\overline{a_1}S_\alpha + \overline{b_1}S^\ast _\alpha)$. We need to analyze some cases:
\begin{itemize}
\item[1.] If $x_1 = 0$ and $x_2 \neq 0$, we have that $y_1 = 0$ and $y_2\neq 0$. Thus $x_2(a_1S_\alpha + b_1S^\ast _\alpha) = y_2(\overline{a_1}S_\alpha + \overline{b_1}S^\ast _\alpha)$, that is, 
\[
\left\lbrace\begin{array}{l}
x_2a_1 = y_2\overline{a_1}\\
x_2b_1 = y_2\overline{b_1}.
\end{array}\right.
\]
From the system above we obtain $\overline{\left( \frac{b_1}{a_1}\right)} = \frac{b_1}{a_1}$, therefore $\frac{b_1}{a_1} \in \mathbb{R}$.
But, if $\frac{b_1}{a_1}\in \mathbb{R}$ we also have that $a_1A_\alpha + b_1A^\ast _\alpha \in L_\alpha \cap \overline{L_\alpha}$. In fact, if $\frac{b_1}{a_1}\in \mathbb{R}$ then $\overline{ \left( \frac{b_1}{a_1}\right) } = \frac{b_1}{a_1} $ which implies that there is $k\in \mathbb{C}$ such that $\left\lbrace\begin{array}{l}
b_1 = k\overline{b_1}\\
a_1 = k\overline{a_1}
\end{array}\right. $. Therefore we can write $a_1A_\alpha + b_1A^\ast _\alpha = k(\overline{a_1}A_\alpha + \overline{b_1}A^\ast _\alpha) \in \overline{L_\alpha}$. So, we can conclude that $\dim (L_\alpha \cap \overline{L_\alpha}) = 2$ when $\frac{b_1}{a_1} \in \mathbb{R}$ and $\dim (L_\alpha \cap \overline{L_\alpha}) = 0$ otherwise.\\

\item[2.] If $x_1 \neq 0$ and $x_2 = 0$, we have $y_1 \neq 0$ and $y_2 = 0$ and it is analogous to the previous item.\\

\item[3.] Now, let us suppose that $x_i,y_i \neq 0$ for $i=1,2$. So $x_1 (a_1A_\alpha+b_1A^\ast _\alpha) + x_2(a_1S_\alpha + b_1S^\ast _\alpha)= y_1 (\overline{a_1}A_\alpha+\overline{b_1}A^\ast _\alpha) + y_2(\overline{a_1}S_\alpha + \overline{b_1}S^\ast _\alpha)$ and then we have
\[
\left\lbrace \begin{array}{l}
x_1b_1 = y_1\overline{b_1}\\
x_1a_1 = y_1\overline{a_1}
\end{array}\right.
\textnormal{ and } \left\lbrace \begin{array}{l}
x_2b_1 = y_2\overline{b_1}\\
x_2a_1 = y_2\overline{a_1}.
\end{array}\right. 
\]
Therefore, we can conclude that $\frac{b_1}{a_1} = \overline{\left( \frac{b_1}{a_1}\right) }$ which implies that $\frac{b_1}{a_1} \in \mathbb{R}$. From the first case we conclude that $\dim (L_\alpha \cap \overline{L_\alpha}) = 2$.
\end{itemize}
  
Finally, observe that when $\dim (L_\alpha \cap \overline{L_\alpha}) = 0$ then $L_\alpha$ defines a generalized complex structure, in this case we have a generalized complex structure of symplectic type (up to $B$-transformation). Due to \cite{VS} we know that $L_\alpha$ reduces to
\[
L_\alpha = \textnormal{span}_\mathbb{C}\{ x_\alpha A_\alpha + (a_\alpha-i)A_\alpha ^\ast, x_\alpha S_\alpha +(a_\alpha-i) S_\alpha ^\ast \},
\]
with $a_\alpha,x_\alpha \in \mathbb{R}$ and $x_\alpha \neq 0$.

As a direct consequence of Proposition \ref{rootdecomposition} and the fact that, for each positive root $\alpha$, the subspace $L_\alpha$ has real index zero or two we have the following:

\begin{theorem}\label{teoDirac}
Let $L$ be an invariant complex Dirac structure on a flag manifold $\mathbb{F}$. Then the real index of $L$ is constant and equal to $2k$ for some $0\leq k \leq l$ where $l$ is the number of positive roots of the associated Lie algebra $\mathfrak{g}$.  
\end{theorem}

\begin{remark}\label{construirDirac}
Let $\mathbb{F}$ be a flag manifold with $\dim\mathbb{F} = 2l$. So, by Theorem \ref{teoDirac}, if we want an invariant complex Dirac structure of real index $k=2r$, where $0\leq r\leq l$, we need to join $r$ Lagrangian subspaces with real index two and $l-r$ Lagrangian subspaces with real index zero, respecting the fact that $\Nij |_L = 0$.
\end{remark}

\begin{remark}
When $r=0$ in Remark \ref{construirDirac} we recover the invariant generalized complex structures described in \cite{VS}.
\end{remark}


In light of Remark \ref{construirDirac}, we will describe the triples of invariant complex Dirac structures $(L_\alpha,L_\beta,L_{\alpha+\beta})$ separating them according to the real index. The possible invariant complex Dirac structures on the subspace $\mathfrak{u}_\alpha$ are described by Proposition \ref{listaDirac}, but the case (D) can be of real index zero or two depending on $a_1$ and $b_1$. So, to simplify the notation we will split these cases. From now on we have:
\begin{align}\label{listaDirac2}
(A)\  & L_\alpha = \textnormal{span}_\mathbb{C}\{ A_\alpha, S_\alpha\} = (\mathfrak{u}_\alpha)_\mathbb{C}; \nonumber \\
(B)\ &  L_\alpha = \textnormal{span}_\mathbb{C}\{ -S^\ast _\alpha, A^\ast _\alpha\}= (\mathfrak{u}_\alpha ^\ast)_\mathbb{C}; \nonumber\\
(C)\ & L_\alpha = \textnormal{span}_\mathbb{C}\{ A_\alpha+i\varepsilon_\alpha S_\alpha, A^\ast _\alpha +i\varepsilon_\alpha S_\alpha ^\ast\} \textnormal{ with } \varepsilon_\alpha = \pm 1 \nonumber\\ 
& \textnormal{(generalized complex structure of complex type)};\\
(D1)\ & L_\alpha = \textnormal{span}_\mathbb{C}\{ a_1A_\alpha+ b_1A^\ast _\alpha, a_1S_\alpha + b_1S^\ast _\alpha\} \textnormal{ with } a_1,b_1 \neq 0 \textnormal{ and } \frac{b_1}{a_1} \in \mathbb{R} \nonumber \\
(D2)\ & L_\alpha = \textnormal{span}_\mathbb{C}\{ x_\alpha A_\alpha+ (a_\alpha-i) A^\ast _\alpha, x_\alpha S_\alpha + (a_\alpha -i) S^\ast _\alpha\} \textnormal{ with } x_\alpha, a_\alpha \in \mathbb{R} \textnormal{ and } x_\alpha \neq 0.\nonumber \\ & \textnormal{(generalized complex structure of symplectic type)} \nonumber
\end{align}
Therefore, now we get that cases (A), (B) and (D1) have real index two and cases (C) and (D2) have real index zero.

From now on, we will use Notation \ref{nottriple} based on the cases appearing in (\ref{listaDirac2}) instead of the cases from Proposition \ref{listaDirac}. For example, let $(\alpha,\beta,\alpha+\beta)$  be a triple of roots and consider $(L_\alpha,L_\beta,L_{\alpha+\beta})$ the complex Dirac structures associated to these roots such that $L_\alpha = \textnormal{span}_\mathbb{C}\{ A_\alpha, S_\alpha\}$, $L_\beta = \textnormal{span}_\mathbb{C}\{ A_\beta+i\varepsilon_\beta S_\beta, A^\ast _\beta +i\varepsilon_\beta S_\beta ^\ast\}$ and $L_{\alpha+\beta} = \textnormal{span}_\mathbb{C}\{ a_{\alpha+\beta}A_{\alpha+\beta}+ b_{\alpha+\beta}A^\ast _{\alpha+\beta}, a_{\alpha+\beta}S_{\alpha+\beta} + b_{\alpha+\beta} S^\ast _{\alpha+\beta}\}$ with $a_{\alpha+\beta},b_{\alpha+\beta} \neq 0$ and $\frac{b_{\alpha+\beta}}{a_{\alpha+\beta}} \in \mathbb{R}$. Then we will refer to the triple $(L_\alpha,L_\beta,L_{\alpha+\beta})$ as $(A,C,D1)$, because $L_\alpha$ is associated to case (A), $L_\beta$ to case (C) and $L_{\alpha+\beta}$ to case (D1).

\subsection{Real index zero}\label{rizero}

The triple $(L_\alpha,L_\beta,L_{\alpha+\beta})$ has real index zero when $L_\alpha$ ,$L_\beta$ and $L_{\alpha+\beta}$ have simultaneously real index zero, that is, they are associated with a generalized complex structure. Then the possible combinations are:

\begin{table}[htb!]
\centering
\begin{tabular}{c|c|c|l}
$L_\alpha$ & $L_\beta$ & $L_{\alpha + \beta}$ & $\Nij = 0$ \\ \hline
C & C & C & $\varepsilon_{\alpha+\beta} - \varepsilon_\alpha - \varepsilon_\beta + \varepsilon_\alpha \varepsilon_\beta \varepsilon_{\alpha+\beta} = 0$\\
C & C & D2 & $\varepsilon_\alpha = -\varepsilon_\beta$\\
C & D2 & C & $\varepsilon_\alpha = \varepsilon_\beta$\\
D2 & D2 & D2 & Equation \ref{Integrabilityconditions}
\end{tabular} 
\end{table}
\FloatBarrier





These are exactly the cases of invariant complex Dirac structures associated with invariant generalized complex structures described in \cite{VS}.

\subsection{Real index two}\label{ritwo}

We are going to analyze triples of subspaces $(L_\alpha,L_\beta,L_{\alpha+\beta})$ such that for two roots we have $\dim (L_\gamma \cap \overline{L_\gamma}) = 0$ and for the other root we have $\dim (L_\gamma \cap \overline{L_\gamma})=2$ where $\gamma \in \{ \alpha,\beta,\alpha+\beta \}$. Observe that the subspaces $L_\gamma$ such that $\dim (L_\gamma \cap \overline{L_\gamma})=0$ are associated to generalized complex structures. Thus we have the following cases:
\begin{enumerate}
\item\label{twocomplex} Two of complex type;
\item\label{twononcomplex} Two of symplectic type (up to $B$-transformation);
\item\label{onecomplexonenoncomplex} One of complex type and one of symplectic type (up to $B$-transformation).
\end{enumerate}

First of all, note that item (\ref{onecomplexonenoncomplex}) cannot happen, because it is not a case presented in Table \ref{tableinv}.

We start by studying item (\ref{twocomplex}), that is, the cases where appear two subspaces associated to invariant generalized complex structures of complex type. Given the triple $(L_\alpha,L_\beta,L_{\alpha+\beta})$ then the subspaces associated with the complex type can be $(L_\alpha,L_\beta)$ and $(L_\beta,L_{\alpha+\beta})$. The case $(L_\alpha,L_{\alpha+\beta})$ coincides with $(L_\beta,L_{\alpha+\beta})$ since $\alpha+\beta = \beta + \alpha$.

\begin{proposition}\label{ccx}
Let $(\alpha,\beta,\alpha+\beta)$ be a triple of roots and consider $(L_\alpha,L_\beta,L_{\alpha+\beta})$ the associated triple of invariant complex Dirac structures such that $L_\alpha$ and $L_\beta$ come from invariant generalized complex structures of complex type, that is, 
\[
L_\gamma = \textnormal{span}_\mathbb{C} \{ A_\gamma + \varepsilon_\gamma i S_\gamma, A^\ast _\gamma + \varepsilon_\gamma iS^\ast _\gamma\},
\]
with $\varepsilon_\gamma = \pm 1$ where $\gamma \in \{\alpha,\beta\}$. Then $\Nij|_L = 0$ if and only if $\varepsilon_\alpha =- \varepsilon_\beta$ for every maximal isotropic subspace $L_{\alpha+\beta}$ such that $\dim_\mathbb{C} (L_{\alpha+\beta} \cap \overline{L_{\alpha+\beta}}) =2$, where $L = L_\alpha \cup L_\beta \cup L_{\alpha+\beta}$.
\end{proposition}
\begin{proof}
Suppose that $L_\alpha$ and $L_\beta$ are invariant complex Dirac structures associated to invariant generalized complex structures of complex type. Thus, $L_\alpha = \textnormal{span}_\mathbb{C} \{ A_\alpha+i\varepsilon_\alpha S_\alpha, A^\ast _\alpha +i\varepsilon_\alpha S^\ast _\alpha\}$ and $L_\beta = \textnormal{span}_\mathbb{C} \{ A_\beta+i\varepsilon_\beta S_\beta, A^\ast _\beta +i\varepsilon_\beta S^\ast _\beta\}$ where $\varepsilon_\alpha,\varepsilon_\beta=\pm 1$. Meanwhile $L_{\alpha+\beta}$ can be any of the three types of structure with real index two, respecting certain conditions so that $\Nij|_L = 0$. Such conditions will tell us the possible combinations of signals $\varepsilon_\alpha$ and $\varepsilon_\beta$. Let us analyze case by case:
\begin{itemize}
\item[a.] If $L_{\alpha+\beta} = \textnormal{span}_\mathbb{C} \{ A_{\alpha+\beta}, S_{\alpha+\beta}\}$, then $\Nij|_L = 0$ if and only if $\varepsilon_\alpha =- \varepsilon_\beta$.

\item[b.] If $L_{\alpha+\beta} = \textnormal{span}_\mathbb{C} \{ -S^\ast _{\alpha+\beta}, A^\ast _{\alpha+\beta}\}$, then $\Nij|_L = 0$ if and only if $\varepsilon_\alpha =- \varepsilon_\beta$.

\item[c.] If $L_{\alpha+\beta} = \textnormal{span}_\mathbb{C} \{ a_{\alpha+\beta} A_{\alpha+\beta}+b_{\alpha+\beta} A^\ast _{\alpha+\beta}, a_{\alpha+\beta} S_{\alpha+\beta}+b_{\alpha+\beta} S^\ast _{\alpha+\beta}\}$ with $\frac{b_{\alpha+\beta}}{a_{\alpha+\beta}}\in \mathbb{R}$, then $\Nij|_L=0$ if and only if 
\begin{eqnarray}\label{sistemasinal}
\left\lbrace \begin{array}{l}
\varepsilon_\beta + \varepsilon_\alpha = 0\\
1 + \varepsilon_\alpha\varepsilon_\beta = 0.
\end{array} \right.
\end{eqnarray}
Note that (\ref{sistemasinal}) is satisfied if and only if $\varepsilon_\alpha =- \varepsilon_\beta$. 
\end{itemize}
\end{proof}

\begin{proposition}\label{xcc}
Let $(\alpha,\beta,\alpha+\beta)$ be a triple of roots and consider $(L_\alpha,L_\beta,L_{\alpha+\beta})$ the associated triple of invariant complex Dirac structures such that $L_\beta$ and $L_{\alpha+\beta}$ come from invariant generalized complex structures of complex type, that is, 
\[
L_\gamma = \textnormal{span}_\mathbb{C} \{ A_\gamma + \varepsilon_\gamma i S_\gamma, A^\ast _\gamma + \varepsilon_\gamma iS^\ast _\gamma\},
\]
with $\varepsilon_\gamma = \pm 1$ where $\gamma \in \{\beta,\alpha+\beta\}$. Then $\Nij|_L = 0$ if and only if $\varepsilon_\beta = \varepsilon_{\alpha+\beta}$ for every maximal isotropic subspace $L_{\alpha}$ such that $\dim_\mathbb{C} (L_{\alpha} \cap \overline{L_{\alpha}}) =2$, where $L = L_\alpha \cup L_\beta \cup L_{\alpha+\beta}$.
\end{proposition}
\begin{proof}
Suppose that $L_\beta$ and $L_{\alpha+\beta}$ are invariant complex Dirac structures associated to invariant generalized complex structures of complex type. Thus, $L_\beta = \textnormal{span} _\mathbb{C}\{ A_\beta+i\varepsilon_\beta S_\beta, A^\ast _\beta +i\varepsilon_\beta S^\ast _\beta\}$ and $L_{\alpha+\beta} = \textnormal{span}_\mathbb{C} \{ A_{\alpha+\beta}+i\varepsilon_{\alpha+\beta} S_{\alpha+\beta}, A^\ast _{\alpha+\beta} +i\varepsilon_{\alpha+\beta} S^\ast _{\alpha+\beta}\}$ where $\varepsilon_\beta , \varepsilon_{\alpha+\beta}=\pm 1$. Meanwhile $L_{\alpha}$ can be any of the three types of structure with real index two, respecting certain conditions so that $\Nij|_L = 0$. Let us analyze case by case:
\begin{itemize}
\item[a.] If $L_{\alpha} = \textnormal{span}_\mathbb{C} \{ A_{\alpha}, S_{\alpha}\}$, then $\Nij|_L = 0$ if and only if $\varepsilon_\beta = \varepsilon_{\alpha+\beta}$.

\item[b.] If $L_{\alpha} = \textnormal{span}_\mathbb{C} \{ -S^\ast _{\alpha}, A^\ast _{\alpha}\}$, then $\Nij|_L = 0$ if and only if $\varepsilon_\beta = \varepsilon_{\alpha+\beta}$.

\item[c.] If $L_{\alpha} = \textnormal{span}_\mathbb{C} \{ aA_{\alpha}+bA^\ast _{\alpha}, aS_{\alpha}+bS^\ast _{\alpha}\}$, then $\Nij|_L=0$ if and only if 
\begin{eqnarray}\label{sistemasinal2}
\left\lbrace \begin{array}{l}
\varepsilon_{\alpha+\beta} - \varepsilon_\beta = 0\\
1 - \varepsilon_\beta\varepsilon_{\alpha+\beta} = 0.
\end{array} \right.
\end{eqnarray}
Note that (\ref{sistemasinal2}) is satisfied if and only if $\varepsilon_\beta = \varepsilon_{\alpha+\beta}$. 
\end{itemize}
\end{proof}

Finally we will study case (\ref{twononcomplex}), that is, the cases where appear two invariant complex Dirac structures associated to invariant generalized complex structures of symplectic type (up to $B$-transformation). We have four cases to analyze, namely:
\begin{itemize}
\item[(2.1)] Suppose $L_\alpha = \textnormal{span}_\mathbb{C} \{ A_\alpha, S_\alpha\}$, $L_\beta = \textnormal{span}_\mathbb{C} \{ x_\beta A_\beta+(a_\beta -i)A^\ast _\beta, x_\beta S_\beta +(a_\beta -i)S^\ast _\beta\}$ and $L_{\alpha+\beta} = \textnormal{span}_\mathbb{C} \{ x_{\alpha+\beta} A_{\alpha+\beta}+(a_{\alpha+\beta} -i)A^\ast _{\alpha+\beta}, x_{\alpha+\beta} S_{\alpha+\beta} +(a_{\alpha+\beta} -i)S^\ast _{\alpha+\beta}\}$. We have $\Nij|_L = 0$ if and only if $\frac{x_\beta}{a_\beta - i} = \frac{x_{\alpha+\beta}}{a_{\alpha+\beta}-i}$ which happens if and only if $x_\beta = x_{\alpha+\beta}$ and $a_\beta = a_{\alpha+\beta}$.

\item[(2.2)] Suppose $L_\alpha = \textnormal{span}_\mathbb{C} \{ x_\alpha A_\alpha+(a_\alpha - i)A^\ast _\alpha, x_\alpha S_\alpha+(a_\alpha -i)S^\ast _\alpha \}$, $L_\beta = \textnormal{span}_\mathbb{C} \{ x_\beta A_\beta+(a_\beta -i)A^\ast _\beta, x_\beta S_\beta +(a_\beta -i)S^\ast _\beta\}$ and $L_{\alpha+\beta} = \textnormal{span}_\mathbb{C} \{ A_{\alpha+\beta}, S_{\alpha+\beta}\}$. We have $\Nij|_L = 0$ if and only if $\frac{x_\alpha}{a_\alpha - i} = -\frac{x_{\beta}}{a_{\beta}-i}$ which happens if and only if $x_\alpha = - x_{\beta}$ and $a_\alpha = a_{\beta}$.

\item[(2.3)] Suppose $L_\alpha = \textnormal{span}_\mathbb{C} \{ a_\alpha A_\alpha +b_\alpha A^\ast _\alpha, a_\alpha S_\alpha+b_\alpha S^\ast _\alpha \}$ with $\frac{b_\alpha}{a_\alpha}\in \mathbb{R}$, $L_\beta = \textnormal{span}_\mathbb{C} \{ x_\beta A_\beta+(a_\beta -i)A^\ast _\beta, x_\beta S_\beta +(a_\beta -i)S^\ast _\beta\}$ and $L_{\alpha+\beta} = \textnormal{span}_\mathbb{C} \{ x_{\alpha+\beta} A_{\alpha+\beta}+(a_{\alpha+\beta} -i)A^\ast _{\alpha+\beta}, x_{\alpha+\beta} S_{\alpha+\beta} +(a_{\alpha+\beta} -i)S^\ast _{\alpha+\beta}\}$. We have $\Nij|_L = 0$ if and only if
\[
a_\alpha x_\beta (a_{\alpha+\beta}-i)-a_\alpha(a_\beta -i)x_{\alpha+\beta} - b_\alpha x_\beta x_{\alpha+\beta} = 0,
\]
which is equivalent to
\begin{eqnarray}\label{twononcomplex2}
x_\beta (a_{\alpha+\beta}-i)-(a_\beta -i)x_{\alpha+\beta} - \frac{b_\alpha}{a_\alpha} x_\beta x_{\alpha+\beta} = 0.
\end{eqnarray}
Separating the real and complex part, we have that (\ref{twononcomplex2}) is satisfied when
\begin{eqnarray}\label{sisttwononcomplex}
\left\lbrace \begin{array}{ll}
x_\beta - x_{\alpha+\beta} = 0\\
x_\beta a_{\alpha+\beta} - a_\beta x_{\alpha+\beta} - \frac{b_\alpha}{a_\alpha} x_\beta x_{\alpha+\beta} = 0.
\end{array}\right.
\end{eqnarray}
From the second expression of (\ref{sisttwononcomplex}) we obtain 
\[
\frac{b_\alpha}{a_\alpha} = \frac{x_\beta a_{\alpha +\beta} - a_\beta x_{\alpha +\beta}}{x_\beta x_{\alpha+\beta}} = \frac{a_{\alpha+\beta}}{x_{\alpha+\beta}} - \frac{a_\beta}{x_\beta}.
\]

\item[(2.4)] Suppose $L_\alpha = \textnormal{span}_\mathbb{C} \{ x_\alpha A_\alpha+(a_\alpha - i)A^\ast _\alpha, x_\alpha S_\alpha+(a_\alpha -i)S^\ast _\alpha \}$, $L_\beta = \textnormal{span}_\mathbb{C} \{ x_\beta A_\beta+(a_\beta -i)A^\ast _\beta, x_\beta S_\beta +(a_\beta -i)S^\ast _\beta\}$ and $L_{\alpha+\beta} = \textnormal{span}_\mathbb{C} \{ a_{\alpha+\beta} A_{\alpha+\beta}+b_{\alpha+\beta} A^\ast _{\alpha+\beta}, a_{\alpha+\beta} S_{\alpha+\beta} + b_{\alpha+\beta} S^\ast _{\alpha+\beta}\}$. We have $\Nij|_L = 0$ if and only if
\[
x_\alpha x_\beta b_{\alpha+\beta}-x_\alpha(a_\beta -i)a_{\alpha+\beta} - (a_\alpha -i) x_\beta a_{\alpha+\beta} = 0.
\]
With an argument analogous to the previous item we obtain 
\begin{eqnarray*}
x_\beta + x_{\alpha+\beta} = 0 \ \textnormal{ and } \ \frac{b_{\alpha+\beta}}{a_{\alpha+\beta}} = \frac{a_\beta}{x_\beta}+\frac{a_\alpha}{x_\alpha}.
\end{eqnarray*}

\end{itemize}

Putting all this information together we have:

\begin{proposition}
Let $(\alpha,\beta,\alpha+\beta)$ be a triple of roots. The associated triple of invariant complex Dirac structures $(L_\alpha,L_\beta,L_{\alpha+\beta})$ of real index two is involutive if it coincides with one of the possibilities that appears in the following table 
\emph{
\begin{table}[htb!]
\begin{center}
\begin{tabular}{c|c|c|l}
$L_\alpha$ & $L_\beta$ & $L_{\alpha + \beta}$ & $\Nij = 0$ \\ \hline
C & C & A {\scriptsize{$\vee$}} B {\scriptsize{$\vee$}} D1 & Proposition \ref{ccx} \\
A {\scriptsize{$\vee$}} B {\scriptsize{$\vee$}} D1 & C & C & Proposition \ref{xcc} \\
A & D2 & D2 & $x_\beta = x_{\alpha+\beta}$ and $a_\beta = a_{\alpha+\beta}$ \\
D2 & D2 & A & $x_\alpha = - x_{\beta}$ and $a_\alpha = a_{\beta}$ \\
D1 & D2 & D2 & $x_\beta = x_{\alpha+\beta}$ and $\frac{b_\alpha}{a_\alpha} = \frac{a_{\alpha+\beta}}{x_{\alpha+\beta}} - \frac{a_\beta}{x_\beta}$\\
D2 & D2 & D1 & $x_\beta = - x_{\alpha+\beta}$ and $\frac{b_{\alpha+\beta}}{a_{\alpha+\beta}} = \frac{a_\beta}{x_\beta}+\frac{a_\alpha}{x_\alpha}$
\end{tabular} 
\end{center}
\end{table}}
\FloatBarrier
\end{proposition}

\begin{notation}
The symbol {\scriptsize{$\vee$}} means, for example, in the first row that we can have the triples $(C,C,A)$, $(C,C,B)$ or $(C,C,D1)$.
\end{notation}

 




\subsection{Real index four}\label{rifour}
Here we are considering triples of subspaces $(L_\alpha,L_\beta ,L_{\alpha+\beta})$ such that for two roots we have $\dim (L_\gamma \cap \overline{L_\gamma}) = 2$ and for the other root we have $\dim (L_\gamma \cap \overline{L_\gamma})=0$ where $\gamma \in \{ \alpha,\beta,\alpha+\beta \}$. That means we have one root associated with an invariant complex Dirac structure that comes from an invariant generalized complex structure. 

Due to Proposition \ref{2dualspaces} one has that triples $(L_\alpha,L_\beta ,L_{\alpha+\beta})$ such that two subspaces are given by $L_\gamma = \textnormal{span}_{\mathbb{C}} \{-S_\gamma ^\ast, A_{\gamma} ^\ast\}$, with $\gamma \in \{\alpha,\beta,\alpha+\beta\}$, and the subspace associated to the other root is of real index zero, then $(L_\alpha,L_\beta ,L_{\alpha+\beta})$ is involutive of real index four.

Note that, according to Table \ref{tableinv}, triples $(L_\alpha,L_\beta ,L_{\alpha+\beta})$ of real index four are necessarily described by the last row. Since we have an extra algebraic condition to analyze, we have the following two cases:
\begin{enumerate}
\item Suppose that $L_\alpha$ is of symplectic type (up to $B$-transformation), that is, $L_\alpha = \textnormal{span}_\mathbb{C} \{ x_\alpha A_\alpha +(a_\alpha - i) A^\ast _\alpha, x_\alpha S_\alpha +(a_\alpha -i)S^\ast _\alpha\}$, $L_\beta = \textnormal{span}_\mathbb{C} \{ a_\beta A_\beta +b_\beta A^\ast _\beta, a_\beta S_\beta +b_\beta S^\ast _\beta \}$ and $L_{\alpha+ \beta} = \textnormal{span}_\mathbb{C} \{ a_{\alpha+\beta} A_{\alpha+\beta} +b_{\alpha+\beta} A^\ast _{\alpha+\beta}, a_{\alpha+\beta} S_{\alpha+\beta} +b_{\alpha+\beta} S^\ast _{\alpha+\beta} \}$ with $\frac{b_\beta}{a_\beta},\frac{b_{\alpha+\beta}}{a_{\alpha+\beta}} \in \mathbb{R}$. In this case, we must have
\[
x_\alpha a_\beta b_{\alpha+\beta} - x_\alpha b_\beta a_{\alpha+\beta} - (a_\alpha -i)a_\beta a_{\alpha+\beta} = 0.
\]
Since $a_\beta,a_{\alpha+\beta} \neq 0$ the last expression can be rewritten as
\[
x_\alpha \frac{b_{\alpha+\beta}}{a_{\alpha+\beta}} - x_\alpha \frac{b_\beta}{a_\beta} - (a_\alpha - i) = 0,
\]
which is an absurd because $x_\alpha,a_\alpha,\frac{b_\beta}{a_\beta},\frac{b_{\alpha+\beta}}{a_{\alpha+\beta}} \in \mathbb{R}$. Therefore this case is not involutive.

\item Suppose $L_{\alpha+\beta}$ of symplectic type (up to $B$-transformation), that is, $L_\alpha = \textnormal{span}_\mathbb{C} \{ a_\alpha A_\alpha +b_\alpha A^\ast _\alpha, a_\alpha S_\alpha +b_\alpha S^\ast _\alpha\}$, $L_\beta = \textnormal{span}_\mathbb{C} \{ a_\beta A_\beta +b_\beta A^\ast _\beta, a_\beta S_\beta +b_\beta S^\ast _\beta \}$ and $L_{\alpha+ \beta} = \textnormal{span}_\mathbb{C} \{ x_{\alpha+\beta} A_{\alpha+\beta} +(a_{\alpha+\beta}-i) A^\ast _{\alpha+\beta}, x_{\alpha+\beta} S_{\alpha+\beta} +(a_{\alpha+\beta}-i) S^\ast _{\alpha+\beta} \}$ with $\frac{b_\alpha}{a_\alpha},\frac{b_\beta}{a_\beta} \in \mathbb{R}$. In this case, we must have
\[
a_\alpha a_\beta (a_{\alpha+\beta} - i) - a_\alpha b_\beta x_{\alpha+\beta} - b_\alpha a_\beta x_{\alpha +\beta} =0.
\]
Since $a_\alpha,a_\beta \neq 0$ the last expression can be rewritten as
\[
(a_{\alpha+\beta}-i) - \frac{b_\beta}{a_\beta}x_{\alpha+\beta} - \frac{b_\alpha}{a_\alpha} x_{\alpha+\beta} = 0,
\]
which cannot happen again because $\frac{b_\alpha}{a_\alpha},\frac{b_\beta}{a_\beta},x_{\alpha+\beta},a_{\alpha+\beta} \in \mathbb{R}$.
\end{enumerate}

Summarizing we obtain the following:

\begin{proposition}
Let $(\alpha,\beta,\alpha+\beta)$ be a triple of roots. The associated triple of invariant complex Dirac structures $(L_\alpha,L_\beta,L_{\alpha+\beta})$ of real index four is involutive if it appears in the following table 

\emph{
\begin{table}[htb!]
\begin{center}
\begin{tabular}{c|c|c|l}
$L_\alpha$ & $L_\beta$ & $L_{\alpha + \beta}$ & $\Nij = 0$ \\ \hline
B & B & C {\scriptsize{$\vee$}} D2 & Proposition \ref{2dualspaces} \\
B & C {\scriptsize{$\vee$}} D2 & B  & Proposition \ref{2dualspaces}
\end{tabular} 
\end{center}
\end{table}}
\FloatBarrier
\end{proposition}

\subsection{Real index six} \label{risix}
This is the last case, here we consider the triples $(L_\alpha,L_\beta,L_{\alpha+\beta})$ where none of the subspaces is of real index zero.

Let $(L_\alpha,L_\beta,L_{\alpha+\beta})$ be a triple with $L_\gamma = \textnormal{span}_\mathbb{C} \{ a_\gamma A_\gamma +b_\gamma A^\ast _\gamma, a_\gamma S_\gamma +b_\gamma S^\ast _\gamma\}$ for all $\gamma \in \{\alpha, \beta,\alpha+\beta \}$. In this case we must have 
\[
a_\alpha a_\beta b_{\alpha+\beta} - a_\alpha b_\beta a_{\alpha+\beta} - b_\alpha a_\beta a_{\alpha+\beta} = 0.
\]
But, since $a_\alpha,a_\beta,a_{\alpha+\beta} \neq 0$ we obtain that
\[
\frac{b_{\alpha+\beta}}{a_{\alpha+\beta}}-\frac{b_\beta}{a_\beta} - \frac{b_\alpha}{a_\alpha} =0,
\]
or equivalently,
\[
\frac{b_{\alpha+\beta}}{a_{\alpha+\beta}}=\frac{b_\beta}{a_\beta} + \frac{b_\alpha}{a_\alpha}.
\]

Therefore the cases of triples with real index six are:

\begin{proposition}
Let $(\alpha,\beta,\alpha+\beta)$ be a triple of roots. The associated triple of invariant complex Dirac structures $(L_\alpha,L_\beta,L_{\alpha+\beta})$ of real index six is involutive if it is one of the possibilities in the following table 

\emph{
\begin{table}[htb!]
\begin{center}
\begin{tabular}{c|c|c|l}
$L_\alpha$ & $L_\beta$ & $L_{\alpha + \beta}$ & $\Nij = 0$ \\ \hline
A & A & A & Proposition \ref{3tangent}\\
B & A & B & Proposition \ref{2dualspaces} \\
B & B & A {\scriptsize{$\vee$}} B {\scriptsize{$\vee$}} D1 & Proposition \ref{2dualspaces} \\
B & D1 & B & Proposition \ref{2dualspaces} \\
D1 & D1 & D1 & $\frac{b_{\alpha+\beta}}{a_{\alpha+\beta}}=\frac{b_\beta}{a_\beta} + \frac{b_\alpha}{a_\alpha}$
\end{tabular} 
\end{center}
\end{table}}
\FloatBarrier

\end{proposition}



Thus we exhausted all possibilities for triples of invariant complex Dirac structures $(L_\alpha,L_\beta,L_\beta)$. 


At this point we are able to prove a converse of Theorem \ref{teoDirac}, that is, we can prove that there exists an invariant complex Dirac structure on $\mathbb{F}$ with constant real index equal to $2k$, with $k \in \{0,1,\cdots ,l\}$ where $l$ is the number of positive roots of the Lie algebra $\mathfrak{g}$ which defines the flag manifold $\mathbb{F}$.

\begin{theorem}
Let $\mathfrak{g}$ be a semisimple Lie algebra and consider $\mathbb{F}$ the maximal flag manifold associated to $\mathfrak{g}$. If $\mathfrak{g}$ has $l$ positive roots, then there exists an invariant complex Dirac structure on $\mathbb{F}$ with constant real index equal to $2k$, where $0\leq k \leq l$.
\end{theorem}

\begin{proof}
Let us denote by $d_i$ the number of positive roots of height $i$, where $d_1$ is the number of simple roots, $m$ is the highest height and $d_1+d_2+\cdots +d_m = l$. We want to construct an invariant complex Dirac structure $L$ with real index $2k$, with $0\leq k \leq l$.

If $k=0$ we have that $L$ comes from a generalized complex structure, whose existence follows from \cite{VS}.

In what follows we look at complex Dirac structures of type A and B according to Proposition \ref{listaDirac}. If $0<k\leq d_1$, then for $k$ simple roots we consider $L_{\alpha_1}, \cdots, L_{\alpha_k}$ of type B, and $L_\beta$ of type A with same sign for all other positive root $\beta$.

If $d_1 < k \leq d_2$, then for all simple root $\alpha$ we consider $L_\alpha$ of type B, for $k-d_1$ roots of height 2 we also consider $L_{\alpha}$ of type B and, for the orther roots, we consider $L_\beta$ of type A with same sign for all other positive root $\beta$.

Proceeding iteratively, if $d_i < k \leq d_{i+1}$ then we consider $L_\alpha$ of type B for all root of height $1,2,\cdots, d_i$, we also consider $L_\alpha$ of type B for $k-(d_1+d_2+\cdots +d_i)$ and, for the other roots, we consider $L_\beta$ of type A with same sign for all other positive root $\beta$.

Now, let us observe that $L$ constructed like above is involutive. In fact, given a triple of roots $(\alpha,\beta, \alpha+\beta)$ then the possible associated triple of subspaces are $(A,A,A)$, $(B,A,A)$, $(A,B,A)$, $(B,B,A)$ and $(B,B,B)$ which are always involutive thanks to the computations done in the previous subsections. In fact, the construction of $L$ does not allow us to have  triples of subspaces $(A,A,B)$, $(B,A,B)$ or $(A,B,B)$, because the root $\alpha+\beta$ is higher than $\alpha$ and $\beta$.
\end{proof}

\section{Classification up to $B$-transformation}\label{sectionBtransformation}



In \cite{AR} was presented a classification up to $B$-transformations of complex Dirac structures in terms of its real index, order and type. Before defining order and type, we need to fix some notation. 

Given $L \subset (V\oplus V^\ast)_ \mathbb{C}$ a complex Dirac structure on $V$, we obtain the following subspaces
\begin{align}\label{notsubspaces}
K = \textnormal{Re} (L\cap \overline{L}) \subseteq V\oplus V^\ast, \ \qquad E = pr_{V_\mathbb{C}}L \subseteq V_ \mathbb{C} \nonumber \\
\Delta = \textnormal{Re}(E\cap \overline{E}) \subseteq V, \ \qquad D=\textnormal{Re}(E+\overline{E}) \subseteq V. 
\end{align}
Note that $E$ is the range and the other subspaces are real defined from $L$. Moreover, recall that there exists a complex $2$-form $\varepsilon\in \wedge^2 E^\ast$ such that 
\[
L = L(E,\varepsilon) = \{ X+\xi \in E\oplus V^\ast \ | \ \xi|_E = \iota_X \varepsilon \}.
\]
Consider the real $2$-form 
\[
\omega_\Delta \colon =\textnormal{Im}(\varepsilon)|_\Delta \in \wedge^2 \Delta^\ast,
\]
where $\textnormal{Im}(\varepsilon)$ means the imaginary part of $\varepsilon$.

Now, we are able to define the type and order:

\begin{definition}
Let $L$ be a linear complex Dirac structure. The {\bf order} of $L$ is defined as 
\[
\textnormal{order}(L) = \textnormal{codim}D.
\] 
\end{definition}

The order is always less than or equal to the real index of the linear complex Dirac structure. In particular, in a generalized complex structure the order is always zero.

\begin{definition}
The {\bf type} of a complex Dirac structure $L$ is 
\[
\textnormal{type}(L)=\textnormal{dim}(E+\overline{E}) - \textnormal{dim}E.
\]
\end{definition}

Once we have fixed the notation we can present the classification done by Agüero and Rubio in \cite{AR}. This classification has as key ingredients the associated presymplectic subspace $(\Delta,\omega_\Delta)$ and a certain complex structure on $D/\Delta$, and states that a complex Dirac structure is (up to $B$-transformations) the product of $(\Delta,\omega_\Delta)$ with the complex structure $D/\Delta$. 

\begin{proposition}[\cite{AR}]
Let $L$ be a complex Dirac structure with real index $r$ and order $s$. Then $L$ is isomorphic to a $B$-transformation of the product of a complex Dirac structure defined by a presymplectic structure with $(r-s)$-dimensional kernel with a complex Dirac structure defined by a codimension-$s$ CR structure.
\end{proposition}

\subsection{Classification on flag manifolds}

We have proved that an invariant complex Dirac structure $L$ on a flag manifold can be decomposed in terms of the root spaces, that is, $L = \oplus_\alpha L_\alpha$ where $L_\alpha$ is a complex structure on $\mathfrak{u}_\alpha$. Moreover, we described every possibility for these subspaces $L_\alpha$ in (\ref{listaDirac2}). Therefore, using the real index, order and type of each  invariant complex Dirac structure presented in (\ref{listaDirac2}), we can classify these structures up $B$-transformation.

\begin{proposition}\label{propBtransfom}
Let $L_\alpha$ be an invariant complex Dirac structure on $\mathfrak{u}_\alpha$. Then, up to $B$-transformations, we have only four possibilities for $L_\alpha$.
\begin{itemize}
\item[(a)] $L_\alpha = (\mathfrak{u}_\alpha)_\mathbb{C}$, that is, $L_\alpha$ is the root space associated to the root $\alpha$;

\item[(b)]  $L_\alpha = (\mathfrak{u}^\ast _\alpha)_\mathbb{C}$, that is, $L_\alpha$ is the dual of the root space associated to the root $\alpha$;

\item[(c)] $L_\alpha = \textnormal{span}_\mathbb{C}\{ A_\alpha + \varepsilon_\alpha iS_\alpha , A_\alpha ^\ast + \varepsilon_\alpha i S^\ast _\alpha \}$ with $\varepsilon_\alpha = \pm 1$, that is, $L_\alpha$ is the $i$-eigenspace of an invariant generalized complex structure on $\mathfrak{u}_\alpha$ of complex type.

\item[(d)] $L_\alpha = \textnormal{span}_\mathbb{C}\{ A_\alpha + \frac{i}{x_\alpha} A_\alpha ^\ast, S_\alpha + \frac{i}{x_\alpha} S^\ast _\alpha \}$ with $x_\alpha,a_\alpha \in \mathbb{R}$ and $x_\alpha \neq 0$, that is, $L_\alpha$ is the $i$-eigenspace of an invariant generalized complex structure on $\mathfrak{u}_\alpha$ of symplectic type.
\end{itemize}
In other words, if we denote by $\textnormal{Dir}_\mathbb{C} (\mathfrak{u}_\alpha)$ the set of all invariant complex Dirac structures on $\mathfrak{u}_\alpha$, then 
\[
\frac{\textnormal{Dir}_\mathbb{C}(\mathfrak{u}_\alpha)}{B\textnormal{-fields}} = \{ (a),(b),(c),(d)\}.
\]
\end{proposition}
\begin{proof}
Note that the subspaces described by Case A and Case D1 in (\ref{listaDirac2}) have the same real index, order and type. In fact, if $L_\alpha = \textnormal{span}_\mathbb{C}\{ A_\alpha,S_\alpha\} = (\mathfrak{u}_\alpha)_\mathbb{C}$, then the subspaces described in (\ref{notsubspaces}) are given by $K = \Delta = D = \mathfrak{u}_\alpha$ and $E = (\mathfrak{u}_\alpha)_\mathbb{C}$. Therefore, we have a complex Dirac structure with real index $2$, order $0$ and type $0$. In the same way, if $L_\alpha = \textnormal{span}_\mathbb{C}\{ a_\alpha A_\alpha + b_\alpha A_\alpha ^\ast, a_\alpha S_\alpha + b_\alpha S^\ast _\alpha \}$ with $\frac{b_\alpha}{a_\alpha} \in \mathbb{R}$, we have that $K = L_\alpha$, $\Delta = D = \mathfrak{u}_\alpha$ and $E = (\mathfrak{u}_\alpha)_\mathbb{C}$. Thus, we have a complex Dirac structure with real index $2$, order $0$ and type $0$.

Since these cases have the same real index, type and order, they are the same subspaces up to $B$-transformation. In fact, we have that Case A is given by $L((\mathfrak{u}_\alpha)_\mathbb{C},0)$ and Case D1 is given by $L((\mathfrak{u}_\alpha)_\mathbb{C}, \varepsilon)$ where $\varepsilon = \frac{b_\alpha}{a_\alpha} (-S^\ast _\alpha)\wedge A^\ast _\alpha$ with $\frac{b_\alpha}{a_\alpha} \in \mathbb{R}$. Then, using the fact that 
\[
e^B\cdot L(E,\varepsilon) = L(E,\varepsilon +\iota ^\ast B),
\]
consider $B \in \wedge^2 (\mathfrak{u}_\alpha)_\mathbb{C} ^\ast$ such that $\iota ^\ast B = \varepsilon$ and we can conclude that 
\[
e^B\cdot L((\mathfrak{u}_\alpha)_\mathbb{C},0) = L((\mathfrak{u}_\alpha)_\mathbb{C},0 +\iota ^\ast B) = L((\mathfrak{u}_\alpha)_\mathbb{C},\varepsilon).
\]

Note that, for Case B in (\ref{listaDirac2}) given by $L_\alpha = \textnormal{span}_\mathbb{C}\{ A_\alpha ^\ast ,S_\alpha ^\ast\} = (\mathfrak{u}_\alpha ^\ast)_\mathbb{C}$, we have $K= \mathfrak{u}_\alpha ^\ast$ and $E = \Delta = D = \{ 0 \}$, obtaining a complex Dirac structure with real index $2$, order $2$ and type $0$. 

Finally, we already know from \cite{GVV} how Case C and Case D2 in (\ref{listaDirac2}) behave under the action of $B$-transformation. Case C are fixed points of $B$-transformations and Case D2 is a $B$-transformation of the $i$-eigenspace of a symplectic structure.
\end{proof}	

\begin{remark}
For each possible case on Proposition \ref{propBtransfom}, we can calculate the invariants given by the real index, order and type, as we can see below.
\begin{table}[htb!]
\begin{center}
\begin{tabular}{c|c|c|c}
 Case & real index & order & type \\ \hline
(a) & 2 & 0 & 0\\
(b) & 2 & 2 & 0  \\
(c) & 0 & 0 & 2  \\
 (d) & 0 & 0 & 0 
\end{tabular} 
\end{center}
\end{table}
\FloatBarrier

Consider $L=\sum_{\alpha >0} L_\alpha$ an invariant complex Dirac structure of real index $2k$, order $2r$ and type $2l$ on a $2n$-dimensional flag manifold $\mathbb{F}$. Then, using the notation of Proposition \ref{propBtransfom}, since the real index of $L$ is $2k$ we must have $k$ invariant complex Dirac structures of cases $(a)$ and $(b)$ and $(n-k)$ invariant complex Dirac structures of cases $(c)$ and $(d)$. Moreover, the order equal to $2r$ gives us that we must have exactly $r$ invariant complex Dirac structures of case $(b)$ and the type $2l$ implies that we have $l$ invariant complex Dirac structures of case $(c)$. Therefore, to obtain an invariant complex Dirac structure $L$ of real index $2k$, order $2r$ and type $2l$ we must have:
\begin{itemize}
\item $k-r$ invariant complex Dirac structures $L_\alpha$ of case $(a)$;
\item $ r$ invariant complex Dirac structures $L_\alpha$ of case $(b)$;
\item $l$ invariant complex Dirac structures $L_\alpha$ of case $(c)$;
\item $n-k-l$ invariant complex Dirac structures $L_\alpha$ of case $(d)$.
\end{itemize}
But it is important to emphasize that the conditions of involutivity presented on Section \ref{involutivity} must always be verified for each triple of roots $(\alpha,\beta,\alpha+\beta)$.
\end{remark}

\section{Examples}\label{secexample}

In this section we will provide some examples of complex Dirac structures for some flag manifolds. Remember that the possible invariant complex Dirac structures on $\mathfrak{u}_\alpha$, where $\alpha$ is a positive root, were described in (\ref{listaDirac2}). In the next examples we consider the flag manifolds associated with the Lie algebras of $\mathfrak{sl}(2,\mathbb{C})$ and $\mathfrak{sl}(3,\mathbb{C})$ due to the low number of roots, which allows us to completely describe the invariant complex Dirac structures in these cases.

\begin{example}
Let $\mathfrak{g} = \mathfrak{sl}(2,\mathbb{C})$, then we have $\Sigma = \Pi^+ = \{ \alpha\}$ and the associated maximal flag manifold is $\mathbb{F}=\mathbb{P}^1$. Since we have only one positive root for $\mathfrak{g}$, then we may have only invariant complex Dirac structures of real index zero or two.

\vspace{.2cm}
\noindent {\bf Real index zero.} If $L_\alpha$ is an invariant complex Dirac structure on $\mathbb{F}$ with real index zero, then $L_\alpha$ is given by an invariant generalized complex structure defined either by a complex structure or by a $B$-transformation of a symplectic form. 

\vspace{.2cm}
\noindent {\bf Real index two.} Now if $L_\alpha$ is an invariant Dirac structure on $\mathbb{F}$ with real index two, then we have three possibilities:
\begin{itemize}
\item The tangent space at the origin, that is, $L_\alpha = (\mathfrak{u}_\alpha)_ \mathbb{C}$; 

\item The cotangent space at the origin, that is, $L_\alpha = (\mathfrak{u}_\alpha ^\ast)_ \mathbb{C}$; 

\item $L_\alpha = \textnormal{span}_\mathbb{C}\{ a_1A_\alpha+ b_1A^\ast _\alpha, a_1S_\alpha + b_1S^\ast _\alpha\}$ with $a_1,b_1 \neq 0$ and $\frac{b_1}{a_1} \in \mathbb{R}$ 
\end{itemize}
We saw in Section \ref{sectionBtransformation} that, up to $B$-transformation, in this case we have only the tangent and cotangent spaces at the origin as invariant complex Dirac structures with real index two.

\end{example}

\begin{example}
 Consider the Lie algebra $\mathfrak{g} = \mathfrak{sl}(3,\mathbb{C})$, then we have $\Sigma = \{ \alpha,\beta\}$ is a simple root system and the corresponding set of positive roots is $\Pi^+ =\{ \alpha,\beta,\alpha+\beta\}$. The maximal flag manifold is given by $\mathbb{F} = \mathbb{F}(1,2)\subset \mathbb{P}^2\times \mathbb{P}^2$ where
 \[
 \mathbb{F}(1,2) = \{ ([x_1,x_2,x_3],[y_1,y_2,y_3])\in \mathbb{P}^2\times \mathbb{P}^2 \ : \ x_1y_1+x_2y_2+x_3y_3 = 0\}.
 \]
We can have invariant Dirac structures on $\mathbb{F}$ with real index zero, two, four or six. Let us present the possibilities according to the real index.

\vspace{.2cm}
\noindent {\bf Real index zero.} These are exactly the invariant complex Dirac structures on $\mathbb{F}$ that come from invariant generalized complex structures. 

\vspace{.2cm}
\noindent {\bf Real index two.} We can see on Section \ref{ritwo} the possible invariant complex Dirac structures on $\mathbb{F}$ with real index two where each case has an involutivity condition involving the signals for structures of type C, the constants $a,b$ for type D1 or the constants $x,a$ for type D2. Then the possible invariant complex Dirac structures are
\begin{table}[htb!]
\begin{center}
\begin{tabular}{c|c|c}
$L_\alpha$ & $L_\beta$ & $L_{\alpha + \beta}$ \\ \hline
A {\scriptsize{$\vee$}} B {\scriptsize{$\vee$}} D1 & C & C \\
A & D2 & D2 \\
D1 & D2 & D2 
\end{tabular} 
\end{center}
\end{table}
\FloatBarrier
\noindent where $\varepsilon_\beta = \varepsilon_{\alpha+\beta}$ for the first row. For $(A,D2,D2)$ we must have $x_\beta = x_{\alpha+\beta}$ and $a_\beta = a_{\alpha+\beta}$ and, finally, for $(D1,D2,D2)$ we must have $x_\beta = x_{\alpha+\beta}$ and $\frac{b_\alpha}{a_\alpha} = \frac{a_{\alpha+\beta}}{x_{\alpha+\beta}} - \frac{a_\beta}{x_\beta}$.
 
\vspace{.2cm}
\noindent {\bf Real index four.} The invariant complex Dirac structures on $\mathbb{F}$ with real index four was described on Section \ref{rifour}. Different from the previous cases, here we do not have algebraic condition for the involutivity, since the cases always contain two invariant complex Dirac structure of type B and these cases were described on Proposition \ref{2dualspaces}. More specifically
 \begin{table}[htb!]
\begin{center}
\begin{tabular}{c|c|c}
$L_\alpha$ & $L_\beta$ & $L_{\alpha + \beta}$ \\ \hline
B & B & C {\scriptsize{$\vee$}} D2 \\
B & C {\scriptsize{$\vee$}} D2 & B 
\end{tabular} 
\end{center}
\end{table}
\FloatBarrier

\vspace{.2cm}
\noindent {\bf Real index six.} In this case the components $L_\delta$ have real index two, which means that none of them come from a generalized complex structure. The possible combinations for $L$ to be involutive are:
\begin{table}[htb!]
\begin{center}
\begin{tabular}{c|c|c}
$L_\alpha$ & $L_\beta$ & $L_{\alpha + \beta}$ \\ \hline
A & A & A \\
B & A & B \\
B & B & A \\
B & B & B \\
B & D1 & B \\
B & B & D1 \\
D1 & D1 & D1
\end{tabular} 
\end{center}
\end{table}
\FloatBarrier
\noindent where the first and fourth cases are the tangent and cotangent spaces at the origin of $\mathbb{F}$, respectively. The cases $(B,A,B)$, $(B,B,A)$, $(B,D1,B)$ and $(B,B,D1)$ are described in Proposition \ref{2dualspaces}. The case $(D1,D1,D1)$ must satisfy $\frac{b_{\alpha+\beta}}{a_{\alpha+\beta}}=\frac{b_\beta}{a_\beta} + \frac{b_\alpha}{a_\alpha}$.
\end{example}


\begin{thebibliography}{99}

\bibitem{A} Aguero, D.: {\it Complex Dirac structures with constant real index}, PhD Thesis, IMPA, 2020.

\bibitem{AR} Aguero, D.; Rubio, R.: {\it Complex Dirac structures: Invariants and Local Structures}. Comm. Math. Phys., {\bf 396} (2022), 623--646. 


\bibitem{Borel} Borel, A.: {\it K\"ahlerian coset spaces of semi-simple Lie groups}. Proc. Nat. Acad. of Sci., {\bf 40} (1954), 1147--1151.

\bibitem{BH1} Borel, A.; Hirzebruch, F.: {\it Characteristic classes and homogeneous spaces, I}. Am. J. Math., {\bf 80.2} (1958), 458--538.

\bibitem{BH2} Borel, A.; Hirzebruch, F.: {\it Characteristic classes and homogeneous spaces, II}. Am. J. Math., {\bf 81.2} (1959), 315--382.


\bibitem{Bu} Bursztyn, H.: {\it A brief introduction to Dirac manifolds}, Geometric and topological methods for quantum field theory (2013), 4--38.



\bibitem{CG}  Cavalcanti, G.: Gualtieri, M., {\it Generalized complex structures on nilmanifolds}, J. Symplectic Geo.,  {\bf 2} (2004) n. 3, 393--410.



\bibitem{C} Courant, T.: {\it Dirac manifolds}, Trans. Amer. Math. Soc., {\bf 319 (2)} (1990), 281--312. 


\bibitem{GVV}   Gasparim, E.: Valencia, F.; Varea, C.:  {\it Invariant generalized complex geometry on maximal flag manifolds and their moduli}, J. Geom. Phys., {\bf 163} (2021), 104108.



\bibitem{G1}   Gualtieri, M.: {\it Generalized complex geometry}, Ann. of Math. (2011), 75--123. 

\bibitem{G2} Gualtieri, M.: {\it Generalized complex geometry}, D.Phil. Thesis, Oxford University, 2003.



\bibitem{H}   Hitchin, N.: {\it Generalized Calabi-Yau manifolds}, Q. J. Math.,  {\bf 54} (2003) n. 3, 281--308.



   
\bibitem{SN}   San Martin, L., Negreiros, C.; {\it Invariant almost Hermitian structures on flag manifolds}, Adv. Math.,  {\bf 178} (2003) n. 2, 277--310.

\bibitem{VS}   Varea, C.; San Martin, L.:  {\it Invariant generalized complex structures on flag manifolds}, J. Geom. Phys., {\bf 150} (2020), 103610.

\bibitem{V}   Varea, C.: {\it Invariant generalized complex structures on partial flag manifolds}, Indag. Math. (N.S.), {\bf 31} (2020) n. 4, 536--555.

 
\end{thebibliography}
\end{document}